\documentclass[11pt,reqno]{amsart}

\usepackage{amssymb,latexsym,graphicx,amscd,upgreek}
\usepackage{lscape}
\usepackage[all]{xy}
\usepackage{color}
\usepackage{epic,eepic}
\usepackage{xspace}
\usepackage{mathrsfs}
\usepackage{setspace}
\usepackage{cases}
\usepackage{bbm}  
\newtheorem{Thm}{Theorem}[section]
\newtheorem{Lem}[Thm]{Lemma}
\newtheorem{Cor}[Thm]{Corollary}
\newtheorem{Prop}[Thm]{Proposition}
\newtheorem{Conj}[Thm]{Conjecture}
\theoremstyle{definition}

\newtheorem{remark}[Thm]{Remark}
\theoremstyle{definition}
\newtheorem{Def}[Thm]{Definition}

\newcommand{\F}{{\mathcal F}}

\newcommand{\M}{{\mathcal M}}
\newcommand{\cO}{{\mathcal O}}
\newcommand{\PP}{{\mathcal P}}

\newcommand{\C}{{\mathbb C}}
\newcommand{\FF}{{\mathbb F}}
\newcommand{\N}{{\mathbb N}}

\newcommand{\Z}{{\mathbb Z}}

\newcommand{\bun}{\mathbf{1}}
\newcommand{\bc}{{\mathbf c}}
\newcommand{\bd}{{\mathbf d}}
\newcommand{\be}{{\mathbf e}}

\newcommand{\bi}{{\mathbf i}}
\newcommand{\br}{{\mathbf r}}

\newcommand{\g}{{\mathfrak g}}

\newcommand{\n}{{\mathfrak n}}

\newcommand{\alp}{\alpha}

\newcommand{\de}{\delta}

\newcommand{\vep}{\varepsilon}
\newcommand{\la}{\lambda}

\newcommand{\cP}{\mathcal{P}}
\newcommand{\cM}{\mathcal{M}}

\newcommand{\cE}{\mathcal{E}}
\newcommand{\cF}{\mathcal{F}}
\newcommand{\cFl}{\mathcal{FL}}

\renewcommand{\le}{\leqslant}

\newcommand{\ad}{\operatorname{ad}}
\newcommand{\diag}{\operatorname{diag}}

\newcommand{\dimv}{\underline{\dim}}

\newcommand{\pdim}{\operatorname{proj.dim}}
\newcommand{\idim}{\operatorname{inj.dim}}

\newcommand{\GL}{\operatorname{GL}}

\newcommand{\Ima}{\operatorname{Im}}
\newcommand{\Ker}{\operatorname{Ker}}

\newcommand{\md}{\operatorname{mod}} 
\newcommand{\rep}{\operatorname{rep}}
\newcommand{\repvp}{\operatorname{rep}_{\mathrm{l.f.}}}
\newcommand{\vp}{\mathrm{l.f.}}

\newcommand{\rk}{\operatorname{rk}}
\newcommand{\rkv}{\underline{{\rm rank}}}

\newcommand{\Hom}{\operatorname{Hom}} 
\newcommand{\Ext}{\operatorname{Ext}}
\newcommand{\End}{\operatorname{End}}
 
\newcommand{\tp}{\operatorname{top}}
\newcommand{\soc}{\operatorname{soc}}

\newcommand{\bil}[1]{\langle #1\rangle}

\newcommand{\df}{\colon\,} 

\def\resp{{\em resp.\ }}

\def\ie{{\em i.e. }}

\newcommand{\bsm}{\begin{smallmatrix}}
\newcommand{\esm}{\end{smallmatrix}}

\newcommand{\bbsm}{\left[\begin{smallmatrix}}
\newcommand{\besm}{\end{smallmatrix}\right]}

\newcommand{\bbm}{\left(\begin{matrix}}
\newcommand{\ebm}{\end{matrix}\right)}

\begin{document}

\title[Quivers with relations for symmetrizable Cartan matrices III]{Quivers with relations for symmetrizable Cartan \\ matrices III: Convolution algebras}

\author{Christof Gei{\ss}}
\address{Christof Gei{\ss}\newline
Instituto de Matem\'aticas\newline
Universidad Nacional Aut{\'o}noma de M{\'e}xico\newline
Ciudad Universitaria\newline
04510 M{\'e}xico D.F.\newline
M{\'e}xico}
\email{christof@math.unam.mx}

\author{Bernard Leclerc}
\address{Bernard Leclerc\newline
Normandie Univ, France\newline
Unicaen, Lmno, F-14032 Caen France\newline
Cnrs, Umr 6139, F-14032 Caen France\newline
Institut Universitaire de France}
\email{bernard.leclerc@unicaen.fr}

\author{Jan Schr\"oer}
\address{Jan Schr\"oer\newline
Mathematisches Institut\newline
Universit\"at Bonn\newline
Endenicher Allee 60\newline
53115 Bonn\newline
Germany}
\email{schroer@math.uni-bonn.de}

\begin{abstract}
We realize the enveloping algebra of the positive part of a semisimple complex Lie algebra as a convolution algebra of constructible functions on module varieties of some Iwanaga-Gorenstein algebras
 of dimension 1.
\end{abstract}

\maketitle

\setcounter{tocdepth}{1}
\numberwithin{equation}{section}
\tableofcontents

\parskip2mm


\section{Introduction}


\subsection{}
Let $Q$ be a finite quiver without oriented cycles. Let $C$ be the symmetric generalized Cartan matrix corresponding
to the undirected graph underlying $Q$, and let $\g=\g(C)$ be the Kac-Moody Lie algebra attached to $C$. 
Kac \cite{K_Inv} has shown that the dimension vectors of the indecomposable representations of $Q$ form the roots of the 
positive part $\n = \n(C)$ of $\g$. 
For quivers $Q$ of finite type, 
Ringel \cite{R0,R1} found a direct construction of the Lie algebra $\n$ itself, and of its enveloping algebra 
$U(\n)$, in terms of the representation theory of $Q$.
He used Hall polynomials counting extensions of representations  
over $\FF_q$, and recovered $U(\n)$ as the Hall algebra of the path algebra $\FF_qQ$ specialized at $q=1$. 
Later, Schofield \cite{S} replaced counting points of varieties over $\FF_q$
by taking the Euler characteristic of complex varieties, and 
extended Ringel's result to an arbitrary quiver (see also \cite{Rie} in the finite type case).
Finally, Lusztig \cite{Lu1} reformulated Schofield's construction and
obtained $U(\n)$ as a convolution algebra of constructible
functions over the affine spaces $\rep(Q,\bd)$ of representations of $\C Q$ with dimension vector $\bd$.

This paper is a first step towards a broad generalization of Schofield's theorem. 
We take for $C$ an arbitrary symmetrizable
generalized Cartan matrix \cite[Sections~1.1 and 2.1]{K}. 
This means that there exists a diagonal matrix $D$ with positive integer diagonal entries 
such that $DC$ is symmetric. The corresponding Kac-Moody Lie algebra $\g=\g(C)$ is called symmetrizable.
With this datum together with an orientation $\Omega$ of the graph naturally attached to $C$,
we have associated in \cite{GLS1} a finite-dimensional algebra $H=H(C,D,\Omega)$ defined by a quiver with relations.
This algebra makes sense over an arbitrary field $K$, but here we fix $K=\C$ so that varieties of $H$-modules
are complex varieties.
When $C$ is symmetric and $D$ is the identity matrix, $H$ is just the path algebra of the
quiver $Q$ corresponding to $C$ and $\Omega$.
In general, it is shown in \cite{GLS1} that $H$ is Iwanaga-Gorenstein of dimension 1, and that its category 
of locally free modules (\ie modules of homological dimension $\le 1$) carries an Euler form
whose symmetrization is given by $DC$. 
By \cite[Proposition~3.1]{GLS2}
the affine varieties $\repvp(H,\br)$ of locally free $H$-modules with rank vector $\br$
are smooth and irreducible.  
By analogy with \cite{S,Lu1}, we then introduce a convolution bialgebra $\M=\M(H)$
of constructible functions on the varieties $\repvp(H,\br)$.
We show that $\M$ is a Hopf algebra isomorphic to the enveloping algebra of the Lie algebra of its 
primitive elements (Proposition~\ref{prop-M-Hopf}). 
Let again $\n$ denote the positive part of the symmetrizable Kac-Moody Lie algebra $\g$.
Our main result is then:

\begin{Thm}\label{Thm-main}
For $H = H(C,D,\Omega)$, $\cM = \cM(H)$ and $\n = \n(C)$
the following hold:
\begin{itemize}

\item[(i)]
There is a surjective Hopf algebra homomorphism
$$
\eta_H\df U(\n) \to \cM.
$$

\item[(ii)]
Assume that $C$ is of Dynkin type.
Then $\eta_H$ is a Hopf algebra isomorphism.

\end{itemize}
\end{Thm}

We conjecture that $\eta_H$ is an isomorphism for all symmetrizable generalized Cartan matrices $C$ and all symmetrizers $D$.

This generalizes Schofield's theorem in two directions. 
Firstly, 
Theorem~\ref{Thm-main} gives 
a new complex geometric construction of $U(\n)$
for non-symmetric Cartan matrices of Dynkin type. 
Secondly, note that if $D$ is a symmetrizer for $C$, then
$kD$ is also a symmetrizer for every $k \ge 1$. 
As $k$ increases, the categories of locally free modules over
$H(C,kD,\Omega)$ become more and more rich and complicated, the dimension of the varieties $\repvp(H(C,kD,\Omega),\br)$ 
increase and their orbit structure gets finer, but at least in the Dynkin case, and conjecturally in all cases, the convolution algebras  
$\M(H(C,kD,\Omega))$ remain the same. 
Thus for every semisimple complex Lie algebra $\g$ we get
an infinite series of exact categories 
$\repvp(H(C,kD,\Omega)), \ (k\ge 1)$
whose convolution algebras 
$\M(H(C,kD,\Omega))$
are isomorphic to $U(\n)$. 
This appears to be new, even when $C$ is symmetric.

Theorem~\ref{Thm-main} is inspired from \cite{S}.
In fact when $C$ is symmetric and $D$ is the identity matrix 
the algebra $\M$ coincides with the algebra denoted by $R^+(Q)$ in \cite{S}.
However in all other cases, 
$\M$ is defined using filtrations which are not composition series.
We also note that, following \cite{Lu1} and in contrast to \cite{S},
we systematically use the language of constructible functions and convolution products.

\subsection{}
Let us outline the structure of the paper.
In Section~\ref{sec-DlabRingel} we briefly recall the definition of a representation of a modulated graph
in the sense of Dlab and Ringel.
In Section~\ref{section-H} we review the definition of $H(C,D,\Omega)$ and the results 
of \cite{GLS1} which will be needed in the sequel. 
In Section~\ref{section-M} we introduce the algebra of constructible functions $\M(H)$ and 
show that it is a homomorphic image of $U(\n)$ (Corollary~\ref{cor-phi}). 
In Section~\ref{section-pseudo} we construct for the preprojective modules over the algebras $H = H(C,D,\Omega)$ analogues of the Auslander-Reiten sequences of the preprojective representations 
of a modulated graph associated with the datum $(C,D,\Omega)$.
This will be used in the proof of our main result, but it should
also be of independent interest.
Section~\ref{sec-primitive} is devoted
to the explicit construction of certain primitive elements
in the Hopf algebra $\cM(H)$ in the case where $C$ is of type $B_n$, $C_n$, $F_4$ or $G_2$, and $D$ is the minimal symmetrizer.  
In Section~\ref{sec-proofmain} we prove
that the homomorphism $U(\n) \to \cM(H)$ is an isomorphism provided $C$ is
of Dynkin type. For a minimal symmetrizer $D$, this follows from Section~\ref{sec-primitive}.
For the case of an arbitrary symmetrizer, we use a geometric result from \cite{GLS2}.
Finally, we discuss some examples in Section~\ref{sec-examples}.

\subsection{Notation}
By $\N$ we denote the natural numbers including $0$.
If not mentioned otherwise, by a \emph{module} we mean a
finite-dimensional left module.
For a module $M$ and a positive integer $m$, let $M^m$ be the
direct sum of $m$ copies of $M$.


\section{Representations of modulated graphs}\label{sec-DlabRingel}


Following Dlab and Ringel \cite{DR} we quickly review the definition of a 
representation of a modulated graph.
Note however that \cite{DR} uses somewhat dual conventions than those presented here.

Let $C = (c_{ij})\in M_n(\Z)$ be a \emph{symmetrizable generalized Cartan matrix}, and let 
$D=\diag(c_1,\ldots,c_n)$ be a \emph{symmetrizer} of $C$. 
This means that $c_i \in \Z_{>0}$, and
\[
c_{ii} = 2, \qquad\qquad
 c_{ij} \le 0\quad\mbox{for}\quad i\not = j,\qquad\qquad
 c_ic_{ij} = c_jc_{ji}.
\]
An \emph{orientation} of $C$ is a subset 
$\Omega \subset  \{ 1,2,\ldots,n \} \times \{ 1,2,\ldots,n \}$
such that the following hold:
\begin{itemize}

\item[(i)]
$\{ (i,j),(j,i) \} \cap \Omega \not= \varnothing$
if and only if $c_{ij}<0$;

\item[(ii)]
For each sequence $((i_1,i_2),(i_2,i_3),\ldots,(i_t,i_{t+1}))$ with
$t \ge 1$ and $(i_s,i_{s+1}) \in \Omega$ for all $1 \le s \le t$, we have
$i_1 \not= i_{t+1}$.

\end{itemize}

Let $F$ be a field, and let $(F_i,{_i}F_j)$ with $1 \le i \le n$ and $(i,j) \in \Omega$ be a \emph{modulation} associated with $(C,D,\Omega)$.
Thus $F_i$ is an $F$-skew-field with $\dim_F(F_i) = c_i$,
and ${_i}F_j$ is an $F_i$-$F_j$-bimodule such that $F$ acts centrally
on ${_i}F_j$, and we have
${_i}F_j \cong F_i^{|c_{ij}|}$ as left $F_i$-modules and 
${_i}F_j \cong F_j^{|c_{ji}|}$ as right $F_j$-modules.
(Such a modulation exists, provided we work over
a suitable ground field $F$.)
A \emph{representation} $(X_i, X_{ij})$ of this modulation consists of a finite-dimensional
$F_i$-module $X_i$ for each $1 \le i \le n$, and of an $F_i$-linear map
$$
X_{ij}\df {_i}F_j \otimes_{F_j} X_j \to X_i
$$
for each $(i,j) \in \Omega$. 
Let 
$$
S := \prod_{i=1}^n F_i
\text{\;\;\; and \;\;\;}
B := \bigoplus_{(i,j) \in \Omega} {_i}F_j.
$$
Thus $B$ is an $S$-$S$-bimodule.
Let $T = T(C,D,\Omega) = T_S(B)$ be the corresponding tensor algebra.
The algebra $T$ is a finite-dimensional hereditary $F$-algebra, and
the abelian category of representations of the modulation $(F_i,{_i}F_j)$ is equivalent
to the category $\md(T)$ of finite-dimensional $T$-modules.
We refer to \cite{DR} for further details on the representation theory
of modulated graphs.


\section{The algebras $H(C,D,\Omega)$}\label{section-H}


\subsection{Definition of $H(C,D,\Omega)$}\label{subsec-defH}
We use the same notation as in \cite{GLS1}. 
Let $(C,D,\Omega)$ be as in Section~\ref{sec-DlabRingel}, i.e. $C$ is a symmetrizable
generalized Cartan matrix, $D = \diag(c_1,\ldots,c_n)$ is a
symmetrizer and $\Omega$ an orientation of $C$.
When $c_{ij} < 0$ define
\[
g_{ij} := |\gcd(c_{ij},c_{ji})|,\qquad
f_{ij} := |c_{ij}|/g_{ij}.
\]
Let
$Q := Q(C,\Omega) := (Q_0,Q_1)$ be the quiver with 
vertex set $Q_0 := \{ 1,\ldots, n\}$ and 
with arrow set  
$$
Q_1 := \{ \alp_{ij}^{(g)}\df j \to i \mid (i,j) \in \Omega, 1 \le g \le g_{ij} \}
\cup  \{ \vep_i\df i \to i \mid 1 \le i \le n \}.
$$
Let
$$
H := H(C,D,\Omega) := \C Q/I
$$ 
where $\C Q$ is the path algebra of $Q$, and $I$ is the ideal of $\C Q$
defined by the following
relations:
\begin{itemize}

\item[(H1)] 
for each $i$, we have 
\[
\vep_i^{c_i} = 0;
\]

\item[(H2)]
for each $(i,j) \in \Omega$ and each $1 \le g \le g_{ij}$, we have
\[
\vep_i^{f_{ji}}\alp_{ij}^{(g)} = \alp_{ij}^{(g)}\vep_j^{f_{ij}}.
\]

\end{itemize}
This definition is illustrated by many examples in \cite[Section~13]{GLS1}.

Clearly, $H$ is a finite-dimensional $\C$-algebra.
It is known \cite[Theorem 1.1]{GLS1} that $H$ is Iwanaga-Gorenstein of dimension 1.
This means that for an $H$-module $M$ the following are equivalent:
\begin{itemize}

\item
$\pdim(M) < \infty$,

\item
$\idim(M) < \infty$,
 
\item
$\pdim(M) \le 1$,

\item 
$\idim(M) \le 1$. 

\end{itemize}
Moreover, if $M$ is a submodule of a projective $H$-module and if $\pdim(M)\le 1$,
then $M$ is projective. 
Dually, if $M$ is a quotient module of an injective $H$-module and $\idim(M)\le 1$
then $M$ is injective.

Note that if $C$ is symmetric and if $D$ is the identity matrix, 
then $H$ is isomorphic to the path algebra $\C Q^\circ$, where
$Q^\circ$ is the acyclic quiver obtained from $Q$ by deleting
all loops $\vep_i$.
More generally, it is easy to see that if $C$ is symmetric
and $D=\diag(k,\ldots,k)$ for some $k>0$, then $H$ is isomorphic 
to $R Q^\circ := R\otimes_\C \C Q^\circ$, where $R$ is the truncated polynomial ring
$\C[x]/(x^k)$. In that case, $H$-modules are nothing else than 
representations of $Q^\circ$ over the ring $R$. 
When $C$ is only symmetrizable, one has a similar picture by replacing
the path algebra $R Q^\circ$ by a generalized modulated graph over a family of truncated polynomial
rings, as we shall now explain.

\subsection{Generalized modulated graphs}\label{subsec-modulation}
Let $H = H(C,D,\Omega)$. 
It was shown in \cite[Section~5]{GLS1} that $H$ gives rise to a generalized modulated graph, and that
the category of $H$-modules is isomorphic to the category of representations of this generalized modulated graph.
This viewpoint, which is very close to Dlab and Ringel's \cite{DR} representation theory of modulated graphs outlined in Section~\ref{sec-DlabRingel},
will be useful in several places below.

For $i = 1,\ldots,n$, let $H_i$ be the subalgebra of $H$ generated by $\vep_i$.
Thus $H_i$ is isomorphic to $\C[x]/(x^{c_i})$.
For $(i,j) \in \Omega$, we define
$$
{_i}H_j := H_i \,\,{\rm Span}_\C(\alp_{ij}^{(g)} \mid 1 \le g \le g_{ij})\,H_j. 
$$
It is shown in \cite{GLS1} that ${_i}H_j$
is an $H_i$-$H_j$-bimodule, which is free as a left $H_i$-module
and free as a right $H_j$-module.
An $H_i$-basis of ${_i}H_j$ is given by
$$
\{ \alp_{ij}^{(g)}, \alp_{ij}^{(g)}\vep_j,\ldots,
\alp_{ij}^{(g)}\vep_j^{f_{ij}-1} \mid 1 \le g \le g_{ij} \}.
$$
In particular, we have an isomorphism ${_i}H_j \cong H_i^{|c_{ij}|}$ of left $H_i$-modules, and we have an isomorphism
${_i}H_j \cong H_j^{|c_{ji}|}$ of right $H_j$-modules.

The tuple $(H_i, {_i}H_j)$ with $1 \le i \le n$ and $(i,j) \in \Omega$ is called a \emph{generalized modulation} associated with the datum
$(C,D,\Omega)$. 
A \emph{representation} $(M_i, M_{ij})$ of this generalized modulation consists of a finite-dimensional
$H_i$-module $M_i$ for each $1 \le i \le n$, and of an $H_i$-linear map
$$
M_{ij}\df {_i}H_j \otimes_{H_j} M_j \to M_i
$$
for each $(i,j) \in \Omega$. 
The representations of this generalized modulation form an abelian category 
$\rep(C,D,\Omega)$ isomorphic to the category of $H$-modules \cite[Proposition 5.1]{GLS1}.  
(Here we identify the category of $H$-modules with the category of
representations of the quiver $Q(C,\Omega)$ satisfying the relations
(H1) and (H2).)
Given a representation $(M_i,M_{ij})$ in $\rep(C,D,\Omega)$ the corresponding $H$-module
$
(M_i,M(\alp_{ij}^{(g)}),M(\vep_i))
$ 
is obtained as follows:
the $\C$-linear map $M(\vep_i)\df M_i \to M_i$ is given by
$$
M(\vep_i)(m) := \vep_i m.
$$
(Here we use that $M_i$ is an $H_i$-module).
For $(i,j) \in \Omega$, the $\C$-linear map $M(\alp_{ij}^{(g)})\df M_j \to M_i$ is defined by
$$
M(\alp_{ij}^{(g)})(m) := M_{ij}(\alp_{ij}^{(g)} \otimes m).
$$
The maps $M(\alp_{ij}^{(g)})$ and $M(\vep_i)$ satisfy the defining relations (H1) and (H2)
of $H$ because the maps $M_{ij}$ are $H_i$-linear.

\subsection{Locally free $H$-modules} \label{subsec-loc-free}
We say that an $H$-module $M=(M_i,M(\alp_{ij}^{(g)}),M(\vep_i))$ is \emph{locally free} if
for every $i$ the $H_i$-module $M_i$ is free. By \cite[Theorem 1.1]{GLS1}, $M$ is locally free
if and only if $\pdim(M) \le 1$.
The full subcategory $\repvp(H)$ whose objects are the finite-dimensional locally free modules
is closed under extensions, kernels of epimorphisms and cokernels of monomorphisms,
and it has Auslander-Reiten sequences \cite[Lemma 3.8, Theorem 3.9]{GLS1}.

The \emph{rank vector} of $M\in \repvp(H)$ is defined as 
$\rkv(M) = (\rk(M_1),\ldots,\rk(M_n))$.
Let $\alpha_1,\ldots,\alpha_n$ be the standard basis of $\Z^n$.
For $1 \le i \le n$ we denote by $E_i$ the unique locally free $H$-module 
with rank vector $\alpha_i$. 
In other words, $E_i$ is nothing else than $H_i$
regarded as an $H$-module in the obvious way. 

For $M,N\in\repvp(H)$, the integer
\[
\bil{M,N}_H := \dim \Hom_H(M,N) - \dim \Ext_H^1(M,N) 
\]
depends only on the rank vectors $\rkv(M)$ and $\rkv(N)$, see \cite[Proposition 4.1]{GLS1}.
The map $(M,N)\mapsto \bil{M,N}_H$ thus descends to a bilinear form on the Grothendieck group $\Z^n$
of $\repvp(H)$, given on the basis $\alp_i=\rkv(E_i)$ by
$$
\bil{\alp_i,\alp_j}_H = 
\begin{cases}
c_ic_{ij} & \text{if $(j,i) \in \Omega$},\\
c_i & \text{if $i=j$},\\
0 & \text{otherwise}.
\end{cases}
$$
Let $(-,-)_H$ 
be the symmetrization of $\bil{-,-}_H$ defined by
$(a,b)_H := \bil{a,b}_H + \bil{b,a}_H$, and let
$q_H$ be the quadratic form defined by
$q_H(a) := \bil{a,a}_H$.
Note that $(-,-)_H$ is nothing else than the symmetric bilinear form   
\[
 (\alp_i,\alp_j) = c_ic_{ij},\qquad (1\le i,j \le n)
\]
associated with the symmetric matrix $DC$.


\section{The convolution algebra $\M$}\label{section-M}


\subsection{Definition of the algebra $\M$}\label{def-M_H}
As before, let $H = H(C,D,\Omega)$ and $Q = Q(C,\Omega)$.
For an arrow $a\df i \to j$ in $Q$ set $s(a) = i$ and $t(a) = j$.
Let $\rep(H,\bd)$ be the affine complex variety of $H$-modules with
dimension vector $\bd = (d_1,\ldots,d_n)$. 
By definition the closed points in $\rep(H,\bd)$ are
tuples
$$
M = (M(a))_{a \in Q_1} \in
\prod_{a \in Q_1} \Hom_\C(\C^{d_{s(a)}},\C^{d_{t(a)}})
$$
of $\C$-linear maps such that
$$
M(\vep_i)^{c_i} = 0
$$
and for each $(i,j) \in \Omega$ and $1 \le g \le g_{ij}$ we have
\[
M(\vep_i)^{f_{ji}}M(\alp_{ij}^{(g)}) = M(\alp_{ij}^{(g)})M(\vep_j)^{f_{ij}}.
\]
The group $G_\bd := \GL_{d_1} \times \cdots \times \GL_{d_n}$ acts on $\rep(H,\bd)$ by conjugation.
The $G_\bd$-orbit of $M \in \rep(H,\bd)$ is denoted by
$\cO_M$.
The $G_\bd$-orbits of $\rep(H,\bd)$ are in one-to-one correspondence with isomorphism classes of $H$-modules with dimension vector $\bd$.

Recall that a \emph{constructible function} on a complex algebraic variety $V$ is a map $\varphi\df V \to \C$ such that
the image of $\varphi$ is finite, and for each $a \in \C$ the preimage $\varphi^{-1}(a)$
is a constructible subset of $V$.
Let $\F_\bd$ be the complex vector space of constructible functions $f\df \rep(H,\bd) \to \C$ which are constant 
on $G_\bd$-orbits, and let
\[
\F = \F(H)= \bigoplus_{\bd\in\N^n} \F_\bd.
\]
We endow $\F$ with a convolution product $*$ defined as in 
\cite[Section~10.12]{Lu1} or \cite[Section~2.1]{Lu2}, using Euler characteristics of 
constructible subsets. Namely, we put
\[
 (f*g)(X) = \int_{Y\subseteq X} f(Y)g(X/Y) d\chi, \qquad (f,g\in\F,\ X\in \rep(H)). 
\]
Here, the integral is taken on the projective variety of all $H$-submodules $Y$ of $X$, and for a constructible function 
$\varphi$ on a variety $V$, we set
\[
 \int_{Y\in V} \varphi(Y) d\chi = \sum_{a\in\C} a \cdot \chi(\varphi^{-1}(a)).
\]
It is well known that $(\F,*)$ has the structure of an $\N^n$-graded associative $\C$-algebra, see e.g. \cite[Section~4.2]{BT}.

Let $\be_i= (0,\ldots,c_i,\ldots,0)\in\N^n$ be the dimension vector of $E_i$. 
Let $\theta_i\in \F$ denote the characteristic function of the $G_{\be_i}$-orbit of $\rep(H,\be_i)$ corresponding to $E_i$. 
In other words, we have
$$
\theta_i(M) = \begin{cases}
1 & \text{if $M \cong E_i$},\\
0 & \text{otherwise}.
\end{cases}
$$
\begin{Def}
We denote by $\M = \M(H)$ the subalgebra of $(\F,*)$ generated by $\theta_1,\ldots,\theta_n$, and for $\bd \in \N^n$ we set
$\M_\bd = \M \cap \F_\bd$. 
\end{Def}

Note that $\M_\bd$ is a finite-dimensional vector space. 
The identity element $\bun_\M$ of $\M$ is the characteristic function of the zero $H$-module.

\begin{Lem}\label{lem-support}
Let $f\in\M_\bd$ and $X\in \rep(H,\bd)$.  
If $X$ is not locally free, then $f(X)=0$.
\end{Lem}

\begin{proof}
For an $H$-module $X$ and a sequence $\bi = (i_1,\ldots,i_k)\in \{1,\ldots,n\}^k$,
we have, by definition of the convolution product $*$, 
\begin{equation}\label{eq.product}
 (\theta_{i_1}*\cdots*\theta_{i_k})(X) = \chi(\cFl_{X,\bi}),
\end{equation}
where $\cFl_{X,\bi}$ is the constructible set of flags 
\[
(0=X_0\subset X_1 \subset \cdots \subset X_k=X) 
\]
of $H$-modules such that
$X_j/X_{j-1} \cong E_{i_j}$ for all $1\le j\le k$.
By \cite[Lemma 3.8]{GLS1} the category of locally free $H$-modules is closed under extensions,
hence if $X$ is not locally free we have $\cFl_{X,\bi} = \varnothing$ for every sequence $\bi$. 
This shows that $(\theta_{i_1}*\cdots*\theta_{i_k})(X) = 0$ for every sequence $\bi$, and
thus, by definition of $\M$, that $f(X)=0$ for every $f\in\M$.
\end{proof}

\begin{remark}
{\rm
When the Cartan matrix $C$ is symmetric and $D$ is the identity matrix, the algebra 
$H$ is the path algebra $\C Q^\circ$ (see Section~\ref{subsec-defH}). In that case
$\M(H)$ coincides with the algebra $R^+(\C Q^\circ)$ of \cite[Section~2.3]{S}, and with the algebra $\M_0(\Omega)$
of \cite[Section~10.19]{Lu1}.
}
\end{remark}

\subsection{Varieties of locally free $H$-modules}
Let $\bd = (d_1,\ldots,d_n) \in \N^n$ be a dimension vector.
If $M\in \rep(H,\bd)$ is locally free, its rank vector is 
$\br = (r_1,\ldots,r_n)$ where $r_i = d_i/c_i$. 
Hence locally free modules can 
only exist if $d_i$ is divisible by $c_i$ for every $i$.
In this case we say that $\bd$ is $\bc$-divisible.
Let $\repvp(H,\br)$ be the union of all $G_\bd$-orbits $\cO_M$ of locally free
modules $M$ with rank vector $\br$.
By Lemma~\ref{lem-support}, the support of every constructible function
$f \in \M_\bd$ is contained in $\repvp(H,\br)$.

\begin{Prop}[{\cite[Proposition~3.1]{GLS2}}]\label{fibredim1}
Let $\bd = (d_1,\ldots,d_n)$ be $\bc$-divisible as above. 
Set $r_i := d_i/c_i$ and  $\br = (r_1,\ldots,r_n)$. 
Then  $\repvp(H,\br)$ is a non-empty open subset of $\rep(H,\bd)$. 
Furthermore,
$\repvp(H,\br)$ is smooth and irreducible of dimension
$\dim(G_\bd) - q_H(\br)$.
\end{Prop}

\subsection{Bialgebra structure of $\M$}
Consider the direct product of algebras $H\times H$. Modules for $H\times H$ are pairs $(X_1,X_2)$
of modules for $H$. 
An $H\times H$-submodule of $(X_1,X_2)$ is a pair $(Y_1,Y_2)$ where $Y_1$ is an $H$-submodule of $X_1$
and $Y_2$ is an $H$-submodule of $X_2$. 
Note that we can regard $H\times H$ as the algebra 
$H(C \oplus C, D \oplus D, \Omega \oplus \Omega)$,
where $C \oplus C$ (resp. $D \oplus D$) means the block diagonal matrix with two diagonal blocks equal to $C$ (resp. $D$), and
$\Omega \oplus \Omega$ is the obvious orientation of $C \oplus C$ induced by the orientation $\Omega$ of $C$.
Therefore we can define as above a convolution algebra $\F(H\times H)$.

We have an algebra embedding of $\F(H)\otimes \F(H)$ into $\F(H \times H)$
by setting
\[
 (f\otimes g)(X,Y) = f(X)g(Y).
\]
Following Ringel \cite{R1} (see also \cite[Section~4.3]{BT}), one introduces a map
$$
c\df \F(H) \to \F(H\times H)
$$ 
by
$c(f)(X,Y) = f(X\oplus Y)$.

\begin{Prop}\label{prop-com}
The map $c$ restricts to an algebra homomorphism
$$
c\df \M(H) \to \M(H) \otimes \M(H)
$$
for the convolution product
such that 
\[
c(\theta_i) = \theta_i\otimes 1 + 1\otimes \theta_i.
\]
This makes $\M(H)$ into a cocommutative bialgebra.
\end{Prop}

\begin{proof} 
We first show that $c$ is a homomorphism from $\F(H)$ to $\F(H\times H)$ 
for the convolution product.
Indeed, on the one hand we have for $f,g\in\F(H)$
\begin{equation}\label{int1}
(c(f*g))(X,Y) = \int_{Z\subseteq X\oplus Y} f(Z)\, g((X\oplus Y)/Z)\, d\chi,
\end{equation}
and on the other hand
\begin{equation}\label{int2}
(c(f)*c(g))(X,Y) = \int_{Z_1\subseteq X,\,Z_2\subseteq Y} f(Z_1\oplus Z_2)\, g((X\oplus Y)/(Z_1\oplus Z_2))\, d\chi. 
\end{equation}
To show that the two integrals are the same, we consider the $\C^*$-action on $X\oplus Y$ given by 
\[
 \la\cdot(x,y) = (\la x,y),\qquad (x\in X,\ y\in Y,\ \la\in\C^*).
\]
This induces a $\C^*$-action on the variety of submodules $Z$ of $X\oplus Y$, whose fixed points
are exactly the submodules of the form $Z=Z_1\oplus Z_2$ with $Z_1\subseteq X$ and $Z_2\subseteq Y$.
Moreover, for a submodule $Z$ of $X\oplus Y$ and $\la\in\C^*$, the $H$-module $\la\cdot Z$ is isomorphic
to $Z$, so for every $f\in\F(H)$ we have $f(\la\cdot Z)= f(Z)$.
Thus, using \cite[Corollary~2]{BB}, we get that
(\ref{int1}) and (\ref{int2}) are equal.
It follows that $c$ restricts to an algebra homomorphism $\M(H) \to \F(H\times H)$.
Since $E_i$ is indecomposable, we have 
\[
 c(\theta_i)(X,Y) = \theta_i(X\oplus Y) 
 = 
  \left\{
 \begin{array}{ll}
 1 & \mbox{if $X\cong E_i$ and $Y=0$, or $X=0$ and $Y\cong E_i$,} \\
 0 & \mbox{otherwise.}
 \end{array}
 \right.
\]
Thus $c(\theta_i)$ can be identified with 
$\theta_i\otimes 1 + 1\otimes \theta_i \in \M(H)\otimes\M(H)\subset \cF(H) \otimes \cF(H) \subset \cF(H \times H)$. 
Finally, since $\M(H)$ is generated by the $\theta_i$'s and $c$ is multiplicative, 
this implies that the image $c(\M(H))$ is indeed contained in
$\M(H)\otimes\M(H)$.
\end{proof}

An element $f$ of $\M$ is called \emph{primitive} if $c(f) = f\otimes1 + 1\otimes f$.

\begin{Lem}\label{lem-indec}
An element $f$ of $\M$ is primitive if and only if $f$ is supported only on indecomposable 
modules.
\end{Lem}

\begin{proof} This follows immediately from  the equality $f(X\oplus Y) = c(f)(X,Y)$.
\end{proof}

It is easy to see that if $f$ and $g$ are primitive, then the Lie bracket
\[
 [f,g]:=f*g-g*f
\]
is also primitive. Hence the subspace $\PP(\M) \subset \M$ of primitive elements has a natural Lie algebra structure.

\begin{Prop}\label{prop-M-Hopf}
$(\M,*,c)$ is a Hopf algebra isomorphic to the universal enveloping algebra $U(\PP(\M))$. 
\end{Prop}

\begin{proof}
A nonzero element $f$ of $\M$ is \emph{group-like} if $c(f)=f\otimes f$.
Arguing as in \cite[Section~4.5]{BT}, we see that the only group-like element is the identity element $\bun_\M$.
Indeed, if $f$ is group-like 
for any $H$-module~$X$ and $k\in\N$ we have $f(X^k) = f(X)^k$.
If $f(X) \not = 0$ for a module $X \not = 0$, then the decomposition of $f$ with respect
to the direct sum $\bigoplus_{\bd}\M_\bd$ has infinitely many nonzero components, a contradiction.  
Hence we have $f=\la\bun_\M$ for some scalar $\la \not= 0$.
Since $f(0 \oplus 0) = f(0)f(0)$, we get that $\la = \la^2$.
This implies $f= \bun_\M$.  

Therefore, we can repeat the last part of the proof of \cite[Theorem]{R1}:
Since $\M$ is a cocommutative coalgebra over $\C$, 
\cite[Lemma 8.0.1]{Sw} together with the previous paragraph implies
that $\cM$ is irreducible, hence a Hopf algebra \cite[Theorem 9.2.2]{Sw}. It then follows from \cite[Theorem 13.0.1]{Sw} that 
$\M$ is isomorphic as a Hopf algebra to the universal enveloping algebra $U(\PP(\M))$ of the Lie algebra $\PP(\M)$.
\end{proof}

\begin{remark}
{\rm
When the Cartan matrix $C$ is symmetric and $D$ is the identity matrix, 
the Lie algebra $\PP(\M)$ coincides with the Lie algebra $L^+(\C Q^\circ)=L^+(Q^\circ)$ of \cite[Section~2.6]{S}.
}
\end{remark}

\subsection{Relations in $\M$}\label{sec-relations}
For $f\in\M$ we denote by $\ad f$ the endomorphism of $\M$ defined by 
\[
\ad f(g) := [f,g],\qquad (g\in\M).
\]

\begin{Prop}\label{prop-relations}
 The generators $\theta_1,\ldots,\theta_n$ of $\M$ satisfy the relations
\[
(\ad \theta_i)^{1-c_{ij}}(\theta_j) = 0
\]
for all $i \not= j$.
\end{Prop}

\begin{proof}
Since $\theta_i \in \PP(\M)$, we have
\[
 \Theta_{ij}:=(\ad \theta_i)^{1-c_{ij}}(\theta_j) \in \PP(\M)
\]
for all $i \not= j$.
By Lemma~\ref{lem-support} and Lemma~\ref{lem-indec}, to check that $\Theta_{ij}=0$ it is therefore sufficient to check that 
there is no indecomposable locally free $H$-module with rank vector
$(1-c_{ij})\alpha_i + \alpha_j$.
Let $M$ be a locally free module $M$ with this rank vector.

Let us assume first that $(i,j)\in\Omega$.
Then, by Section~\ref{subsec-modulation}, $M$ is given by an $H_i$-linear map 
\[
M_{ij}\df {_i}H_j\otimes_{H_j}M_j \to M_i,
\]
where $M_j = H_j$ and $M_i = H_i^{1-c_{ij}}$.
Now, ${_i}H_j\otimes_{H_j}M_j \cong {_i}H_j$ is a free $H_i$-module of rank $-c_{ij}$,
so $M_i$ contains a direct summand $N_i$ isomorphic to $H_i$ such that $N_i\cap\Ima(M_{ij}) = 0$. 
It follows that $M$ has a direct summand isomorphic to $E_i$, and therefore $M$ is not indecomposable.

For $(j,i) \in \Omega$, $M$ is given by an $H_j$-linear
map 
$$
M_{ji}\df {_j}H_i \otimes_{H_i} M_i \to M_j.
$$
Let 
$$
M_{ji}^\vee\df M_i \to {_i}H_j \otimes_{H_j} M_j
$$
be the associated adjoint map as defined in \cite[Section~5]{GLS1}.
The map $M_{ji}^\vee$ is $H_i$-linear, and as $H_i$-modules we have
$M_i = H_i^{1-c_{ij}}$ and ${_i}H_j \otimes_{H_j} M_j \cong H_i^{-c_{ij}}$.
Similarly as before, $M_i$ contains a direct summand $N_i$ isomorphic to
$H_i$ such that $N_i \cap \Ker(M_{ji}^\vee) = 0$.
It follows that $M$ has a direct summand isomorphic to $E_i$.
So again, $M$ is not indecomposable.
\end{proof}


Let $\g = \g(C)$ be the symmetrizable Kac-Moody Lie algebra over $\C$ with Cartan matrix $C$. It is defined by the 
following presentation. There are $3n$ generators $e_i,\ f_i,\ h_i\ (1\le i\le n)$ subject to the relations:
\begin{itemize}
 \item[(i)]
 $[e_i,f_j] = \de_{ij}h_i$;
 \item[(ii)]
 $[h_i,h_j]=0$; 
 \item[(iii)]
 $[h_i,e_j]=c_{ij}e_j,\quad [h_i,f_j]=-c_{ij}f_j$;
 \item[(iv)] 
 $(\ad e_i)^{1-c_{ij}}(e_j) = 0, \quad (\ad f_i)^{1-c_{ij}}(f_j) = 0$\qquad $(i\not = j)$.
\end{itemize} 
Let $\n = \n(C)$ be the Lie subalgebra generated by $e_i\ (1\le i\le n)$.
Then $U(\n)$ is the associative $\C$-algebra with generators $e_i\ (1\le i\le n)$ subject to the 
relations
\[
 (\ad e_i)^{1-c_{ij}}(e_j) = 0, \qquad (1 \le i\not = j \le n).
\]

\begin{Cor}\label{cor-phi}
The assignment $e_i \mapsto \theta_i$ extends to a surjective Hopf algebra homomorphism 
\[
\eta_H\df U(\n) \to \M.
\]
\end{Cor}

\begin{proof}
The algebras $U(\n)$ and $\cM$ are generated as algebras by their subsets
$\n$ and $\cP(\cM)$ of primitive elements, respectively.
It follows from Proposition~\ref{prop-relations} that 
$e_i \mapsto \theta_i$ extends to a surjective algebra homomorphism
$\eta_H\df U(\n) \to \cM$.
The comultiplication of $U(\n)$ and $\cM$ are given by 
$e_i \mapsto e_i \otimes 1 + 1 \otimes e_i$ and
$\theta_i \mapsto \theta_i \otimes 1 + 1 \otimes \theta_i$,
respectively.
The antipodes of $U(\n)$ and $\cM$ are given by 
$x \mapsto -x$ and $m \mapsto -m$ for all $x \in \n$ and all $m \in \cP(\cM)$,
respectively.
It follows that $\eta_H$ is a Hopf algebra homomorphism.
\end{proof}

\subsection{Weight spaces}\label{sec-weight}
As before, 
let $\alpha_1,\ldots,\alpha_n$ be the standard basis of $\Z^n$.
As before, let $H = H(C,D,\Omega)$ and $\n = \n(C)$.
The algebras $U(\n)$ and $\cM = \cM(H)$ are both $\N^n$-graded via 
$\deg(e_i) := \alpha_i$ and $\deg(\theta_i) := \alpha_i$, respectively.
For $\alpha \in \N^n$ let $U(\n)_\alpha$ and $\cM_\alpha$ be the vector spaces of
elements of degree $\alpha$.
Define
$$
\n_\alpha := \n \cap U(\n)_\alpha
\text{\;\;\; and \;\;\;}
\cP(\cM)_\alpha := \cP(\cM) \cap \cM_\alpha.
$$
It follows that 
$$
\n = \bigoplus_{\alpha \in \N^n} \n_\alpha
\text{\;\;\; and \;\;\;}
\cP(\cM) = \bigoplus_{\alpha \in \N^n} \cP(\cM)_\alpha.
$$
Furthermore, the surjective Hopf algebra homomorphism
$$
\eta_H\df U(\n) \to \cM
$$
restricts to a surjective Lie algebra homomorphism
$\n \to \cP(\cM)$ and to a 
surjective linear map $\eta_{H,\alpha}\df \n_\alpha \to \cP(\cM)_\alpha$
for each $\alpha \in \N^n$.
The \emph{weight space} $\n_\alpha$ is non-zero if and only if
$\alpha$ is a positive root of the Kac-Moody Lie algebra $\g = \g(C)$.
In particular, if $C$ is a Cartan matrix of Dynkin type, then we have
$$
\dim(\n_\alpha) = 
\begin{cases}
1 & \text{if $\alpha \in \Delta^+(C)$},\\
0 & \text{otherwise}.
\end{cases}
$$


\section{Pseudo Auslander-Reiten sequences for preprojective
modules}\label{section-pseudo}


\subsection{Auslander-Reiten translates and Coxeter transformation}
Let $C \in M_n(\Z)$ 
be a symmetrizable generalized Cartan matrix, let $D = \diag(c_1,\ldots,c_n)$
be a symmetrizer of $C$, and let
$\Omega$ be an orientation of $C$.

Let $T = T(C,D,\Omega)$ be the tensor algebra of some modulation associated with $(C,D,\Omega)$, compare Section~\ref{sec-DlabRingel}.
Up to isomorphism there are $n$ simple $T$-modules 
$S_1^T,\ldots,S_n^T$ with $\dim \End_T(S_i^T) = c_i$.  
For $1 \le i \le n$, let
$P_i^T$ (resp. $I_i^T$) be the indecomposable projective (resp. injective) $T$-module with $\tp(P_i^T) \cong S_i^T$ (resp. $\soc(I_i^T) \cong S_i^T)$.
Let $c_T\df \Z	^n \to \Z^n$ be the Coxeter transformation of $T$.
The automorphism $c_T$ can be defined by the rule
$$
c_T(\dimv(P_i^T)) = -\dimv(I_i^T)
$$
for all $1 \le i \le n$.

Let $H = H(C,D,\Omega)$.
Up to isomorphism there are $n$ simple $H$-modules 
$S_1^H,\ldots,S_n^H$ corresponding to the vertices of $Q(C,D)$.  
For $1 \le i \le n$, let
$P_i^H$ (resp. $I_i^H$) be the indecomposable projective (resp. injective) $H$-module with $\tp(P_i^H) \cong S_i^H$ (resp. $\soc(I_i^H) \cong S_i^H)$.
Let $c_H\df \Z^n \to \Z^n$ be the Coxeter transformation of $H$,
as defined in \cite[Section~2.5]{GLS1}.
The automorphism $c_H$ can be described by the rule
$$
c_H(\rkv(P_i^H)) = -\rkv(I_i^H)
$$
for all $1 \le i \le n$, see \cite[Section~3.4]{GLS1}.

For $M,N \in \repvp(H)$ and $X,Y \in \md(T)$, 
let 
\[
\bil{M,N}_H := \dim \Hom_H(M,N) - \dim \Ext_H^1(M,N)
\]
and
\[
\bil{X,Y}_T := \dim \Hom_T(X,Y) - \dim \Ext_T^1(X,Y).
\]
Then \cite[Lemmas~3.2 and 3.3 and Section~4]{GLS1} imply the following crucial result.

\begin{Prop}\label{prop-dimrank}
We have
\begin{itemize}

\item[(i)]
$\dimv(P_i^T) = \rkv(P_i^H)$;

\item[(ii)]
$\dimv(I_i^T) = \rkv(I_i^H)$;

\item[(iii)]
$c_T = c_H$;

\item[(iv)]
For $X,Y \in \md(T)$ and $M,N \in \repvp(H)$ with $a = (a_1,\ldots,a_n) = \dimv(X) = \rkv(M)$ and $b = (b_1,\ldots,b_n) = \dimv(Y) = \rkv(N)$
we have
$$
\bil{X,Y}_T = \bil{M,N}_H = 
\sum_{i=1}^n c_ia_ib_i + \sum_{(j,i) \in \Omega} c_ic_{ij}a_ib_j. 
$$

\end{itemize}
\end{Prop}
By Proposition~\ref{prop-dimrank}(iv) we can consider 
$\bil{-,-}_H$ and $\bil{-,-}_T$ as 
bilinear forms on $\Z^n \times \Z^n$.

Let $\tau_T$ (\resp $\tau_T^-$) denote the Auslander-Reiten translation (\resp the inverse Auslander-Reiten translation) 
of the algebra $T$.
By definition $\tau_T = {\rm D}{\rm Tr}$ is the dual of the transpose
${\rm Tr}$.
We refer to \cite[Chapter~4]{ARS} for further details.
The next lemma follows from general Auslander-Reiten theory and from the fact that $T$ is a hereditary algebra \cite[Section~4]{R2}.

\begin{Lem}\label{lemma-pseudo1}
For non-projective indecomposable $T$-modules $X$ and $Y$ 
the following hold:
\begin{itemize}

\item[(i)]
$\Hom_T(X,Y) \cong \Hom_T(\tau_T(X),\tau_T(Y))$;

\item[(ii)]
$\Ext_T^1(X,Y) \cong \Ext_T^1(\tau_T(X),\tau_T(Y))$;

\item[(iii)]
$\dimv(\tau_T(X)) = c_T(\dimv(X))$.

\end{itemize}
\end{Lem}

There is an obvious dual of Lemma~\ref{lemma-pseudo1} for non-injective indecomposable $T$-modules
$X$ and $Y$.

Let $\tau_H$ (\resp $\tau_H^{-}$) denote the Auslander-Reiten translation (\resp the inverse Auslander-Reiten translation)
of the algebra $H$.
Recall from \cite{GLS1} that an indecomposable $H$-module $M$ is called $\tau$-\emph{locally free} 
provided $\tau_H^m(M)$ is locally free for all $m \in \Z$.

\begin{Lem}\label{lemma-pseudo2}
For non-projective indecomposable $\tau$-locally free $H$-modules
$M$ and $N$ the following hold:
\begin{itemize}

\item[(i)]
$\Hom_H(M,N) \cong \Hom_H(\tau_H(M),\tau_H(N))$;

\item[(ii)]
$\Ext_H^1(M,N) \cong \Ext_H^1(\tau_H(M),\tau_H(N))$;

\item[(iii)]
$\rkv(\tau_H(M)) = c_H(\rkv(M))$.

\end{itemize}
\end{Lem}

\begin{proof}
Parts (i) and (iii) are a consequence of the results in 
\cite[Section~11.1]{GLS1}.
Part (ii) follows then from \cite[Proposition~4.1]{GLS1}.
\end{proof}

There is an obvious dual of Lemma~\ref{lemma-pseudo1} for non-injective indecomposable $\tau$-locally free $H$-modules
$M$ and $N$.

\subsection{Preprojective modules}\label{sec-preproj}
As before, for $1 \le i \le n$
let $P_i^T$ (resp. $P_i^H$) be the indecomposable projective 
$T$-module (resp. $H$-module) associated with $i$.
Recall that an indecomposable  $T$-module $X$ (resp. an indecomposable $H$-module $M$) is \emph{preprojective}
if $X \cong \tau_T^{-k}(P_i^T)$ (resp. $M \cong \tau_H^{-k}(P_i^H)$)
for some $1 \le i \le n$ and some $k \ge 0$.
By \cite[Proposition 11.6]{GLS1}, preprojective modules are locally free.

\begin{Prop}\label{prop-preproj1}
For all $k \ge 0$ we have
$$
\dimv(\tau_T^{-k}(P_i^T)) = \rkv(\tau_H^{-k}(P_i^H)).
$$
\end{Prop}

\begin{proof}
This follows immediately from Proposition~\ref{prop-dimrank}
and Lemmas~\ref{lemma-pseudo1} and \ref{lemma-pseudo2}.
\end{proof}

\begin{Prop}\label{prop-preproj3}
Let $M = \tau_H^{-k}(P_i^H)$ and
$N = \tau_H^{-s}(P_j^H)$ be preprojective $H$-modules, and let
$X = \tau_T^{-k}(P_i^T)$ and
$Y = \tau_T^{-s}(P_j^T)$ be the corresponding preprojective 
$T$-modules.
Then we have 
\begin{align*}
\dim \Hom_H(M,N) &= \dim \Hom_T(X,Y),\\
\dim \Ext_H^1(M,N) &= \dim \Ext_T^1(X,Y).
\end{align*}
\end{Prop}

\begin{proof}
By Proposition~\ref{prop-dimrank} we have 
$\bil{\rkv(M),\rkv(N)}_H = \bil{\dimv(X),\dimv(Y)}_T$.
Since 
\begin{align*}
\bil{\rkv(M),\rkv(N)}_H &= \dim \Hom_H(M,N) - \dim \Ext_H^1(M,N),\\
 \bil{\dimv(X),\dimv(Y)}_T &= \dim \Hom_T(X,Y) - \dim \Ext_T^1(X,Y),
\end{align*}
it is enough to show that
$\dim \Hom_H(M,N) = \dim \Hom_T(X,Y)$.

If $k \le s$, then 
\begin{align*}
\dim \Hom_H(M,N) &= \dim \Hom_H(P_i^H,\tau_H^{-(s-k)}(P_j^H))\\
&= [\tau_H^{-(s-k)}(P_j^H):S_i^H]\,\dim \End_H(S_i^H)\\
&= [\tau_H^{-(s-k)}(P_j^H):S_i^H]\\
&= \dimv(\tau_H^{-(s-k)}(P_j^H))_i\\
&= c_i(\rkv(\tau_H^{-(s-k)}(P_j^H)))_i,
\end{align*}
and
\begin{align*}
\dim \Hom_T(X,Y) &= \dim \Hom_T(P_i^T,\tau_T^{-(s-k)}(P_j^T))\\
&= [\tau_T^{-(s-k)}(P_j^T):S_i^T]\,\dim \End_T(S_i^T)\\
&= c_i(\dimv(\tau_H^{-(s-k)}(P_j^T)))_i.
\end{align*}
If $k > s$, then
\begin{align*}
\Hom_H(M,N) &\cong \Hom_H(\tau_H^{-(k-s)}(P_i^H),P_j^H) = 0,
\end{align*}
and
\begin{align*}
\Hom_T(X,Y) &\cong \Hom_T(\tau_T^{-(k-s)}(P_i^T),P_j^T) = 0.
\end{align*}
This finishes the proof.
\end{proof}

\begin{Cor}
Let $M = \tau_H^{-k}(P_i^H)$ and
$N = \tau_H^{-s}(P_j^H)$ be preprojective $H$-modules.
Then we have
\begin{align*}
\dim \Hom_H(M,N) &= c_i(\rkv(\tau_H^{-(s-k)}(P_j^H)))_i,\\
\dim \Ext_H^1(M,N) &= c_j(\rkv(\tau_H^{-(k-s-1)}(P_i^H)))_j.
\end{align*}
\end{Cor}

\begin{proof}
The first equality follows from the proof of Proposition~\ref{prop-preproj3}.
The second equality follows from combining the first equality with the Auslander-Reiten formula
$$
\dim \Ext_H^1(M,N) = \dim \Hom_H(\tau_H^{-1}(N),M).
$$
(Since $N$ is a module of injective dimension at most one, we don't need to consider stable homomorphism spaces.)  
\end{proof}

The dimension vectors (resp. rank vectors) of the preprojective
$T$-modules (resp. $H$-modules) are positive roots of the Kac-Moody
Lie algebra $\g(C)$. (For $H$-modules, see \cite[Lemma 3.2 and Proposition 11.5]{GLS1}.)
A positive root $\alpha \in \Delta^+$ is $H$-\emph{preprojective}
if $\alpha$ is the dimension vector of a preprojective $T$-module
(or equivalently the rank vector of a preprojective $H$-module).
In this case let $X(\alpha)$ be the preprojective $T$-module with
dimension vector $\alpha$, and let $M(\alpha)$ be the preprojective
$H$-module with rank vector $\alpha$.

As a direct consequence of Proposition~\ref{prop-preproj3}
we get the following result.

\begin{Cor}\label{cor-preproj2}
Let $\gamma_1,\ldots,\gamma_t$ be $H$-preprojective positive roots.
Then
$
X(\gamma_1) \oplus \cdots \oplus X(\gamma_t)
$
is rigid if and only if
$
M(\gamma_1) \oplus \cdots \oplus M(\gamma_t)
$
is rigid.
\end{Cor}

Given $H$, the rank vectors of the modules $P_i^H$ are known \cite[Lemma 3.2]{GLS1}.
For any $H$-preprojective positive root $\alpha$, we have
$\alpha = c_H^{-k}(\rkv(P_i^H))$ for uniquely determined $1 \le i \le n$ and $k \ge 0$.
Thus the computation of the dimensions of Hom and Ext spaces
for preprojective modules is purely combinatorial.

A tuple of indecomposable $H$-modules $(M_1,\ldots,M_t)$ is
a \emph{cycle} if there exists a non-zero non-isomorphism $M_i \to M_{i+1}$ for all $1 \le i \le t-1$ and $M_1 \cong M_t$.
Such a cycle is a \emph{strict cycle} if additionally $M_i \not\cong M_{i+1}$ for all $1 \le i \le t-1$.

As a consequence of Proposition~\ref{prop-preproj3} and the corresponding well known statement for preprojective $T$-modules, we get the following result.

\begin{Cor}\label{prop-preproj4}
There are no strict cycles consisting of preprojective $H$-modules.
\end{Cor}

\subsection{Example}
Let 
$H = H(C,D,\Omega)$ be defined by the quiver
$$
\xymatrix{
1 \ar@(ul,ur)^{\vep_1}& \ar[l]^{\alpha_{12}} 2\ar@(ul,ur)^{\vep_2} & 
\ar[l] 3 & \ar[l] 4
}
$$
with relations $\vep_1^2 = \vep_2^2 = 0$ and $\vep_1\alpha_{12}
= \alpha_{12}\vep_2$.
Thus $H$ is of Dynkin type $F_4$.
Let $T = T(C,D,\Omega)$ be the corresponding hereditary algebra
over some suitable ground field $F$.
The Auslander-Reiten quiver of $T$ 
is shown in Figure~\ref{F4C1}.
The vertices of the quiver are the positive roots and stand for the
preprojective $T$-modules $X(\alpha)$.
(Since $C$ is of Dynkin type, all indecomposable $T$-modules
are preprojective.)
For each non-projective $X(\alpha)$ there is a dashed arrow
$\xymatrix{\tau_T(X(\alpha)) & X(\alpha) \ar@{-->}[l]}$.
Recall that we have $c_T(\dimv(X(\alpha))) = \dimv(\tau_T(X(\alpha)))$ and
$c_H(\rkv(M(\alpha))) = \rkv(\tau_H(M(\alpha)))$.
We can use now Figure~\ref{F4C1} to compute for example
the dimension of $\Hom_H(M(2,3,4,2),M(1,2,3,2))$.
We have $\tau_H^2(M(2,3,4,2)) \cong M(1,1,0,0) = P_2^H$ and
$\tau_H^2(M(1,2,3,2)) \cong M(1,2,2,1)$.
Thus 
\begin{align*}
\dim \Hom_H(M(2,3,4,2),M(1,2,3,2))
&=
\dim \Hom_H(P_2^H,M(1,2,2,1)) \\
&= [M(1,2,2,1):S_2^H]\\
&= c_2(\rkv(M(1,2,2,1)))_2\\
&= 2(1,2,2,1)_2 = 4.
\end{align*}
%

\subsection{Geometry of extension varieties}\label{extgeometry}
Let $H = H(C,D,\Omega)$.
Let $M$ and $N$ be rigid locally free $H$-modules with
$\Ext_H^1(N,M) = 0$.
Let
$\bd = (d_1,\ldots,d_n) = \dimv(M) + \dimv(N)$
and $\br = (r_1,\ldots,r_n)$ with $r_i = d_i/c_i$ for all $i$.
Let $\cE(M,N)$ be the set of all $E \in \rep(H,\bd)$ such that
there exists a short exact sequence
$$
0 \to N \to E \to M \to 0.
$$
It follows from \cite[Theorem~1.3]{CBS} that $\cE(M,N)$ is irreducible and open
in $\rep(H,\bd)$.
Furthermore, since $\repvp(H)$ is closed under extensions, we know
that $\cE(M,N)$ is contained in $\repvp(H,\br)$.
By \cite[Proposition~3.1]{GLS2} 
we know that $\repvp(H,\br)$ is irreducible and open in $\rep(H,\bd)$.
Thus $\cE(M,N)$ is open and dense in $\repvp(H,\br)$.

Suppose now additionally, that $\repvp(H,\br)$ contains a rigid 
module $R$.
Since $\cE(M,N)$ is $G_\bd$-stable, open and dense in $\repvp(H,\br)$ and since the $G_\bd$-orbit of
a rigid module is also open, we get that $R \in \cE(M,N)$.
In other words, there exists a short exact sequence
$$
0 \to N \to R \to M \to 0.
$$

\subsection{Construction of pseudo Auslander-Reiten sequences}
As before, let $H = H(C,D,\Omega)$.
Let $\alpha$ be an $H$-preprojective positive root such that
$\alpha$ is not the rank vector of a projective $H$-module.
The Auslander-Reiten sequence in $\md(T)$ ending in $X(\alpha)$
is of the form
$$
0 \to \tau_T(X(\alpha)) \to \bigoplus_{i=1}^t X(\gamma_i)^{m_i}
\to X(\alpha) \to 0
$$
where $\gamma_1,\ldots,\gamma_t$ are pairwise different $H$-preprojective roots and $m_i \ge 1$ for all $i$.

\begin{Thm}\label{thm-pseudoAR}
There is a short exact sequence
$$
0 \to \tau_H(M(\alpha)) \to \bigoplus_{i=1}^t M(\gamma_i)^{m_i}
\to M(\alpha) \to 0
$$
of $H$-modules.
\end{Thm}

\begin{proof}
We have $\Ext_T^1(\tau_T(X(\alpha)),X(\alpha)) = 0$, and therefore
$\Ext_H^1(\tau_H(M(\alpha)),M(\alpha)) = 0$.
From Section~\ref{extgeometry} we know that $\cE(M(\alpha),\tau_H(M(\alpha)))$ is $G_\bd$-stable, open and dense in $\repvp(H,\br)$ where
$\bd = (d_1,\ldots,d_n) = \dimv(M(\alpha)) + \dimv(\tau_H(M(\alpha)))$
and $\br = (r_1,\ldots,r_n)$ with $r_i = d_i/c_i$ for all $i$.
Since $\bigoplus_{i=1}^t X(\gamma_i)^{m_i}$ is a rigid $T$-module,
we know that $R := \bigoplus_{i=1}^t M(\gamma_i)^{m_i}$ is a rigid locally free $H$-module with dimension vector $\bd$. 
Thus the $G_\bd$-orbit of $R$ is open and dense in $\repvp(H,\br)$.
It follows that $R \in \cE(M(\alpha),\tau_H(M(\alpha)))$.
This finishes the proof.
\end{proof}

The short exact sequence appearing in Theorem~\ref{thm-pseudoAR}
is called a \emph{pseudo Auslander-Reiten sequence}, or short
a \emph{pseudo AR-sequence}.


Let 
$$
0 \to X \xrightarrow{f} Y \xrightarrow{g} Z \to 0
$$ 
be a pseudo AR-sequence of preprojective $H$-modules, and let $M$ be an indecomposable $\tau$-locally free $H$-module 
which is not isomorphic to $Z$.
Let $h\df M \to Z$ be a non-zero homomorphism.
By \cite[Lemma~11.7 and Theorem~11.10]{GLS1} this implies that $M$ is preprojective.
By the Auslander-Reiten formula and since $M$ is locally free, we get
that $\Ext_H^1(M,X) \cong D\Hom_H(\tau_H^{-1}(X),M)$.
By Proposition~\ref{prop-preproj4} and the fact that $\tau_H^{-1}(X) \cong Z$ we get $\Ext_H^1(M,X) = 0$.
Thus $h$ factors through $g$.
Similarly, if $M$ is an indecomposable $\tau$-locally free $H$-module 
which is not isomorphic to $X$, and if $h\df X \to M$ is a non-zero homomorphism, then $h$ factors through $f$.
This shows that pseudo Auslander-Reiten sequences share at least some properties with Auslander-Reiten sequences.
However, note that for genuine Auslander-Reiten sequences, the above factorization properties hold
for \emph{every} indecomposable module $M$, not just for the $\tau$-locally free ones.


\section{Construction of primitive elements}\label{sec-primitive}


Let $C$ be a Cartan matrix of Dynkin type, let $D$ be the minimal
symmetrizer of $C$, and let $\Omega$ be an orientation for $C$.
Let $H = H(C,D,\Omega)$ and $\cM = \cM(H)$.
Recall that $\cP(\cM)$ denotes the Lie algebra of primitive elements
in $\cM$.
Let $\Delta^+ = \Delta^+(C)$ be the set of positive roots of the 
Lie algebra $\g(C)$.
For $\gamma \in \Delta^+$ let
$\cP(\cM)_\gamma$
be the subspace of $\cP(\cM)$ of elements of degree $\gamma$.
The aim of Section~\ref{sec-primitive} is to prove the following result, which implies that $\cP(\cM)_\gamma$ is non-zero.

\begin{Thm}\label{thm-primitive}
For each $\gamma \in \Delta^+$ there exists a primitive
element $\theta_\gamma \in \cP(\cM)_\gamma$ such that
$\theta_\gamma(M(\gamma)) = 1$.
\end{Thm}

We prove Theorem~\ref{thm-primitive} for each non-symmetric $C$ of Dynkin type by explicitly constructing
in Sections~\ref{primitive-B_n}, \ref{primitive-C_n}, 
\ref{primitive-F_4}, \ref{primitive-G_2} the elements $\theta_\gamma$. 
For symmetric $C$ of Dynkin type, $H$ is a path algebra of finite type and for each positive root $\gamma$
there is a unique indecomposable $H$-module of dimension vector $\gamma$. Therefore our claim follows from
Schofield's theorem which shows that $\mathfrak{n}$ is identified with the primitive elements of $\cM$.

\subsection{Admissible triples}\label{adm-triples}

A triple 
$(\alpha,\beta,\gamma)$ of positive roots is an $H$-\emph{admissible triple} if the following hold:
\begin{itemize}

\item[(c1)]
$\Hom_H(M(\alpha),M(\beta)) = 0$ and
$\Hom_H(M(\beta),M(\alpha)) = 0$;

\item[(c2)]
There exists a short exact sequence
$$
0 \to M(\alpha) \xrightarrow{f} M(\gamma) \xrightarrow{g} M(\beta)
\to 0;
$$

\item[(c3)]
For each  indecomposable locally free submodule $U$  
of $M(\gamma)$ with $\rkv(U) = \alpha$ we have
$U \cong M(\alpha)$, or for each 
indecomposable locally free factor module $V$  
of $M(\gamma)$ with $\rkv(V) = \beta$ we have
$V \cong M(\beta)$.

\end{itemize}

\begin{Lem}\label{strategy1}
Assume that $(\alpha,\beta,\gamma)$ is a triple of positive roots satisfying {\rm (c2)}.
Then the following hold:
\begin{itemize}

\item[(i)]
$M(\gamma)$ does not have any locally free submodule $U$ with
$\rkv(U)=\beta$;
 
\item[(ii)]
$M(\gamma)$ does not have any locally free factor module $V$ with
$\rkv(V) = \alpha$.

\end{itemize}
\end{Lem}

\begin{proof}
Assume $U$ is a submodule of $M(\gamma)$ with
rank vector $\beta$.
We have 
$$
\bil{\beta,\beta}_H = \dim \Hom_H(M(\beta),U) -
\dim \Ext_H^1(M(\beta),U) > 0.
$$
Thus we get $\Hom_H(M(\beta),U) \not= 0$.
But this implies $\Hom_H(M(\beta),M(\gamma)) \not= 0$,
a contradiction, since there are no strict cycles consisting of
preprojective $H$-modules, see Proposition~\ref{prop-preproj4}.
Part (ii) is proved dually.
\end{proof}

\begin{Lem}\label{strategy2}
Assume that $(\alpha,\beta,\gamma)$ is a triple of positive roots
satisfying {\rm (c1)} and {\rm (c2)}.
Then the following hold:
\begin{itemize}

\item[(i)]
$M(\gamma)$ has only one submodule isomorphic to $M(\alpha)$,
namely $\Ima(f)$;

\item[(ii)]
For any locally free submodule 
$U$ of $M(\gamma)$ with
$\rkv(U) = \alpha$ and $U \not\cong M(\alpha)$ we have
$$
\Hom_H(U,M(\beta)) \not= 0;
$$

\item[(iii)]
$M(\gamma)$ has only one factor module isomorphic to
$M(\beta)$, namely $M(\gamma)/\Ima(f)$;

\item[(iv)]
For any locally free factor module $V$ of $M(\gamma)$ with
$\rkv(V) = \beta$ and $V \not\cong M(\beta)$ we have 
$$
\Hom_H(M(\alpha),V) \not= 0.
$$

\end{itemize}
\end{Lem}

\begin{proof}
Let $f'\df U \to M(\gamma)$ be a monomorphism where $\rkv(U) = \alpha$. 
Suppose that $g\circ f' = 0$.
Then for dimension reasons we get $\Ima(f) = \Ima(f')$. 
If $g \circ f' \not= 0$, then $U \not\cong M(\alpha)$ since
$\Hom_H(M(\alpha),M(\beta)) = 0$.
This yields (i) and (ii).
Parts (iii) and (iv) are proved dually.
\end{proof}

\begin{Prop}\label{strategy3}
Let $(\alpha,\beta,\gamma)$ be an $H$-admissible triple.
Let $\rho,\sigma \in \cP(\cM)$ such that
$\rho(M(\alpha)) = 1$ and $\sigma(M(\beta)) = 1$.
Then $[\rho,\sigma] \in \cP(\cM)$ satisfies
$[\rho,\sigma](M(\gamma)) = 1$. 
\end{Prop}

\begin{proof}
We have
$[\rho,\sigma] = \rho * \sigma
-\sigma * \rho$.
Now $M(\gamma)$ has exactly one indecomposable submodule $U$ with rank vector $\alpha$ and we have $U \cong M(\alpha)$ and $M(\gamma)/U \cong M(\beta)$, or 
$M(\gamma)$ has exactly one indecomposable factor module $V = M(\gamma)/U$ with rank vector $\beta$ and we have $V \cong M(\beta)$ and 
$U \cong M(\alpha)$.
In both cases, we get 
$(\rho * \sigma)(M(\gamma)) = 1$.
Furthermore, $M(\gamma)$ has no submodule with rank vector $\beta$.
Thus we have 
$(\sigma * \rho)(M(\gamma)) = 0$.
This finishes the proof.
\end{proof}

\subsection{Convention for displaying modules}\label{sec-convention}
Let $H = H(C,D,\Omega)$ with $C$ of Dynkin type.
Thus the quiver $Q = Q(C,\Omega)$ has no multiple arrows.
In the following sections we will construct explicitly some
$H$-modules $M$ by giving a $\C$-basis $B_i$ of $M_i$ for
each vertex $i$ of $Q$
together with the action of the arrows of $H$ on the basis 
$B := B_1 \cup \cdots \cup B_n$ of $M$.
This information is displayed in a diagram.
The vertices of the diagram with label $i$ correspond to the elements in $B_i$, and the edges in the diagram show how the arrows of $Q$ act on $B$.
For example,
let
$$
C = \left(\begin{matrix} 
2&-1&&&\\-1&2&-1&&\\&-1&2&-1&\\&&-1&2&-1\\&&&-2&2\end{matrix}\right)
$$
be a Cartan matrix of type $B_5$.
The minimal symmetrizer of $C$ is $D = \diag(2,2,2,2,1)$.
Let $\Omega$ be an orientation of $C$ with
$\{ (2,3),(4,5) \} \subset \Omega$.
The diagram
$$
\xymatrix@-1.8ex{
2 \ar[d]& 3 \ar[d]\ar[l]\ar@{-}[r] & 4\ar[d] & 5 \ar[l] 
\\
2 & 3 \ar[l]\ar@{-}[r] & 4 & 5 \ar[l]\ar[dl]
\\
& 3 \ar[d] & 4 \ar[d]\ar@{-}[l]
\\
& 3 \ar@{-}[r] & 4
}
$$
defines a locally free $H$-module $M$ with $\rkv(M) = (0,1,2,2,2)$.
Each vertex $i$ of the diagram stands for a basis vector of 
$M_i$.
An oriented edge like $\xymatrix{2 & 3 \ar[l]}$ means that
$\alpha_{23}$ sends the basis vector labeled by $3$ to the one labeled by $2$.
A non-oriented edge like $\xymatrix{3 \ar@{-}[r] & 4}$ means that
$\alpha_{34}$ sends the basis vector labeled by $4$ to the one labeled by $3$ provided $(3,4) \in \Omega$, or that $\alpha_{43}$ sends
the basis vector labeled by $3$ to the one labeled by $4$ in case $(4,3) \in \Omega$.
The two arrows starting in the lower $5$ mean that $\alpha_{45}$
sends the basis vector labeled by the lower $5$ to the sum of the
basis vectors labeled by the two $4$'s in the middle.

\subsection{How do we know that we constructed the correct module?}
Let $C$ and $D$ be as in the example in Section~\ref{sec-convention},
and set $\Omega = \{ (1,2),(2,3),(3,4), (4,5) \}$.
Let $H = H(C,D,\Omega)$.
Let $M$ be the $H$-module defined by the diagram
$$
\xymatrix@-1.8ex{
2 \ar[d]& 3 \ar[d]\ar[l] & 4\ar[l]\ar[d] & 5 \ar[l] 
\\
2 & 3 \ar[l] & 4 \ar[l]& 5 \ar[l]\ar[dl]
\\
& 3 \ar[d] & 4 \ar[l]\ar[d]
\\
& 3 & 4 \ar[l]
}
$$
Clearly, $M$ is a locally free $H$-module with rank vector $\br := (0,1,2,2,2)$.
(Just check that all defining relations for $H$ are satisfied and that 
the $H_i$-module $M_i$ is free for all $i$.
This is straightforward.)
Furthermore, we claim that $M$ is isomorphic to the indecomposable preprojective $H$-module $M(\gamma)$ with
$\gamma = \br$.
To prove this, we need to show that $M$ is rigid, compare \cite[Theorem~1.2]{GLS1}.
Let $\bd = \dimv(M) = (0,2,4,4,2)$.
Recall that 
$q_H(\br) = \dim \End_H(M) - \dim \Ext_H^1(M,M)$.
Thus $M$ is rigid if and only if
$\dim \End_H(M) = q_H(\br)$.
The proof of this equality is
done by an explicit calculation which we now carry out for the above example.
In all other examples below, this is done similarly and is left to the reader.

We have $M = (M_i,M(\alpha_{ij}),M(\vep_i))$ with
$(i,j) \in \Omega$ and $1 \le i \le 5$.
We have $M_1 = 0$, $M_2 = K^2$, $M_3 = K^4$, $M_4 = K^4$ and $M_5 = K^2$. 
Numbering the basis vectors in each column of the above diagram
from bottom to top we get 
\begin{align*}
M(\vep_1) &= 0, & 
M(\vep_2) &= 
\bbm 0&1\\0&0\ebm, &
M(\vep_3) &= M(\vep_4) =
\bbm 0&1&0&0\\0&0&0&0\\0&0&0&1\\0&0&0&0\ebm, &
M(\alpha_{12}) &= 0,
\end{align*}
\begin{align*}
M(\alpha_{23}) &= 
\bbm 0&0&1&0\\0&0&0&1\ebm, &
M(\alpha_{34}) &= 
\bbm 1&0&0&0\\0&1&0&0\\0&0&1&0\\0&0&0&1\ebm, &
M(\alpha_{45}) &= 
\bbm 0&0\\1&0\\1&0\\0&1\ebm.
\end{align*}
An endomorphism of $M$ is given by a tuple $F = (F_1,F_2,F_3,F_4,F_5)$
with $F_1 = 0$, $F_2,F_5 \in M_2(K)$ and $F_3,F_4 \in M_4(K)$ such that the
following relations hold:
\begin{align}
F_iM(\vep_i) &= M(\vep_i)F_i\text{\;\;\;\;\;\;\;\;\; for } i = 2,3,4, \label{eq-5.1}
\\
F_2M(\alpha_{23}) &= M(\alpha_{23})F_3,\label{eq-5.2}
\\
F_3M(\alpha_{34}) &= M(\alpha_{34})F_4,\label{eq-5.3}
\\
F_4M(\alpha_{45}) &= M(\alpha_{45})F_5.\label{eq-5.4}
\end{align}
Equation~(\ref{eq-5.3}) implies that $F_3 = F_4$.
From (\ref{eq-5.1}) we get that $F_2,\ldots,F_5$ are of the form
\begin{align*}
F_2 = \bbm a_2&b_2\\0&a_2\ebm,&&
F_3 = F_4 = \bbm a_3&b_3&c_3&d_3\\0&a_3&0&c_3\\e_3&f_3&g_3&h_3\\
0&e_3&0&g_3\ebm,&&
F_5 = \bbm a_5&b_5\\c_5&d_5\ebm.
\end{align*}
with $a_i,b_i,c_j,d_j,e_3,f_3,g_3,h_3 \in K$ for $i= 2,3,5$ and $j=3,5$.
Now (\ref{eq-5.2}) yields $e_3=f_3=0$, $a_2=g_3$ and
$b_2=h_3$.
Equation~(\ref{eq-5.4}) implies that 
$c_3=b_2=-b_3=b_5$, $d_3=c_5=0$ and $a_2=a_3=a_5=d_5$.
Combining all equations we get
\begin{align*}
F_2 = \bbm a_2&b_2\\0&a_2\ebm,&&
F_3 = F_4 = \bbm a_2&-b_2&b_2&0\\0&a_2&0&b_2\\0&0&a_2&b_2\\
0&0&0&a_2\ebm,&&
F_5 = \bbm a_2&b_2\\0&a_2\ebm.
\end{align*}
This shows that $\dim \End_H(M) = 2 = q_H(\br)$.
Thus $M$ is rigid, and therefore $M \cong M(\gamma)$.

\subsection{Filtrations}
Let $\Phi := (\beta_1,\ldots,\beta_t)$ be a sequence of positive roots.
We say that an $H$-module $M$ has a \emph{filtration of type}
$\Phi$ if there is a chain 
$$
0 = M_0 \subset M_1 \subset \cdots \subset M_{t-1} \subset M_t = M
$$
of submodules $M_i$ of $M$ such that $M_i/M_{i-1}$ is locally free
with
$\rkv(M_i/M_{i-1}) = \beta_i$ for all $1 \le i \le t$.

Suppose that $\theta_{\beta_i} \in \cM_{\beta_i}$ for $1 \le i \le t$.
If $\theta_{\beta_1} \cdots \theta_{\beta_t}(M) \not= 0$ for some $H$-module $M$, then
$M$ needs to have a filtration of type $(\beta_1,\ldots,\beta_t)$.
This obvious fact will be used at various places in the following sections.

\subsection{Type $B_n$}\label{primitive-B_n}
Let 
$$
C = \left(
\begin{matrix}
2&-1&&&\\-1&2&\ddots&&\\&\ddots&2&-1&\\&&-1&2&-1\\&&&-2&2
\end{matrix}
\right)
$$
be a Cartan matrix of type $B_n$.
The minimal symmetrizer of $C$ is $D = \diag(2,\ldots,2,1)$.
Let $\Omega$ be an orientation of $C$ such that
$(n-1,n) \in \Omega$.
(The case $(n,n-1) \in \Omega$ is treated dually.)
Set $H = H(C,D,\Omega)$.
We construct now explicitly for each $\gamma \in \Delta^+$
the indecomposable
preprojective $H$-module $M(\gamma)$
with rank vector $\gamma$ and the corresponding primitive element $\theta_\gamma$.

\subsubsection{}
For $1 \le k \le s \le n-1$ and $\gamma = \sum_{i=k}^s \alpha_i$, 
we have
$$
M(\gamma)\df
\xymatrix{
k \ar@{-}[r]\ar[d] & k+1 \ar[d]\ar@{-}[r] & \cdots \ar@{-}[r] & s \ar[d] 
\\
k \ar@{-}[r]&k+1 \ar@{-}[r] & \cdots \ar@{-}[r] & s
}
$$
If $k=s$, set $\theta_\gamma  := \theta_k$.
Thus let $k < s$.
For $(k+1,k) \in \Omega$, set
$\alpha = \sum_{i=k+1}^s \alpha_i$ and $\beta = \gamma-\alpha
= \alpha_k$,
and for $(k,k+1) \in \Omega$, let
$\alpha = \alpha_k$ and $\beta = \gamma-\alpha = \sum_{i=k+1}^s \alpha_i$.
In both cases it is now easy to check that $(\alpha,\beta,\gamma)$ is
an $H$-admissible triple.
Set $\theta_\gamma := [\theta_\alpha,\theta_\beta]$.

\subsubsection{}
For $1 \le k \le n-1$ and $\gamma = \sum_{i=k}^n\alpha_i$,
we have
$$
M(\gamma)\df
\xymatrix{
k \ar@{-}[r]\ar[d] & \cdots \ar@{-}[r] & n-1 \ar[d] & n \ar[l]
\\
k \ar@{-}[r]&  \cdots \ar@{-}[r] & n-1
}
$$
Take $\alpha = \sum_{i=k}^{n-1} \alpha_i$ and $\beta = \gamma-\alpha = \alpha_n$.
It is easy to check that $(\alpha,\beta,\gamma)$ is
an $H$-admissible triple.
Define
$\theta_\gamma := [\theta_\alpha,\theta_\beta]$.

\subsubsection{}
For $1 \le k \le n-1$ and $\gamma = (\sum_{i=k}^{n-1} \alpha_i) + 2\alpha_n$, we have 
$$
M(\gamma)\df
\xymatrix{
k \ar@{-}[r]\ar[d] & \cdots \ar@{-}[r] & n-1 \ar[d] & n \ar[l]
\\
k \ar@{-}[r]& \cdots \ar@{-}[r] & n-1 & n \ar[l]
}
$$
Take $\alpha = \sum_{i=k}^{n-1} \alpha_i$ and $\beta = \alpha_n$, and define
$\theta_\gamma := 1/2[[\theta_\alpha,\theta_\beta],\theta_\beta]$.
It is straightforward to check that $M(\gamma)$ does not have any
filtrations of type $(\beta,\alpha,\beta)$ or $(\beta,\beta,\alpha)$.
We get $\theta_\gamma(M(\gamma)) = 1$.

\subsubsection{}
For $1 \le k \le s \le n-2$ and
$\gamma = \sum_{i=k}^s \alpha_i + \sum_{j=s+1}^n 2\alpha_j$,
we have
$$
M(\gamma)\df
\xymatrix{
k \ar[d] \ar@{-}[r] & \cdots \ar@{-}[r] & s \ar[d] & s+1 \ar[l]\ar[d]\ar@{-}[r]
& \cdots \ar@{-}[r] & n-1 \ar[d] & n \ar[l]
\\
k \ar@{-}[r] & \cdots \ar@{-}[r] & s & s+1 \ar[l]\ar@{-}[r] & \cdots \ar@{-}[r]
& n-1 & n \ar[l]\ar[dl] 
\\
&&&s+1 \ar[d]\ar@{-}[r]
& \cdots \ar@{-}[r] & n-1 \ar[d]
\\
&&&s+1 \ar@{-}[r]
& \cdots \ar@{-}[r] & n-1
}
$$
if $(s,s+1) \in \Omega$, and
$$
M(\gamma)\df
\xymatrix{
&&& s+1 \ar[d]\ar@{-}[r]
& \cdots \ar@{-}[r] & n-1 \ar[d] & n \ar[l]
\\
&&&s+1 \ar@{-}[r] & \cdots \ar@{-}[r]
& n-1 & n\ar[l]\ar[dl]
\\
k \ar[d] \ar@{-}[r] & \cdots \ar@{-}[r] & s \ar[d]\ar[r] & s+1 \ar[d]\ar@{-}[r]
& \cdots \ar@{-}[r] & n-1 \ar[d]
\\
k \ar@{-}[r] & \cdots \ar@{-}[r] & s \ar[r] & s+1 \ar@{-}[r] & \cdots \ar@{-}[r]
& n-1
}
$$
if $(s+1,s) \in \Omega$.
Let $t \ge s+1$ be maximal such that there is a path
$p$ in $Q$ with $s(p) = s+1$ and $t(p) = t$.
For $(s,s+1) \in \Omega$, 
set $\alpha = \sum_{i=s+1}^t \alpha_i$ and $\beta = \gamma-\alpha$,
and for $(s+1,s) \in \Omega$, 
set $\alpha = \sum_{i=k}^t \alpha_i$ and $\beta = \gamma-\alpha$.
In both cases it is now easy to check that $(\alpha,\beta,\gamma)$ is
an $H$-admissible triple.
Set $\theta_\gamma := [\theta_\alpha,\theta_\beta]$.

\subsection{Type $C_n$}\label{primitive-C_n}
Let 
$$
C = \left(
\begin{matrix}
2&-1&&&\\-2&2&-1&&\\&-1&2&\ddots&\\&&\ddots&2&-1\\&&&-1&2
\end{matrix}
\right)
$$
be a Cartan matrix of type $C_n$.
The minimal symmetrizer of $C$ is $D = \diag(2,1,\ldots,1)$.
Let $\Omega$ be an orientation of $C$.
Let $H = H(C,D,\Omega)$.
We construct now explicitly for each $\gamma \in \Delta^+$
the indecomposable
preprojective $H$-module $M(\gamma)$
with rank vector $\gamma$ and the corresponding primitive element $\theta_\gamma$.

\subsubsection{}
For $2 \le k \le s \le n$ and $\gamma = \sum_{i=k}^s \alpha_i$,
we have
$$
M(\gamma)\df
\xymatrix{
k \ar@{-}[r]
& k+1  \ar@{-}[r] & \cdots \ar@{-}[r] & s 
}
$$
For $(k+1,k) \in \Omega$, let
$\alpha = \sum_{i=k+1}^s \alpha_i$ and
$\beta = \gamma-\alpha = \alpha_k$, and for
$(k,k+1) \in \Omega$, let
$\alpha = \alpha_k$ and
$\beta = \gamma-\alpha = \sum_{i=k+1}^s \alpha_i$.
In both cases it is easy to check that $(\alpha,\beta,\gamma)$ is
an $H$-admissible triple.
Set
$\theta_\gamma := [\theta_\alpha,\theta_\beta]$.

\subsubsection{}
For $1 \le k \le n$ and $\gamma = \alpha_1 + \sum_{i=2}^k 2\alpha_i$,
we have
$$
M(\gamma)\df
\xymatrix{
1 \ar[d] \ar@{-}[r]& 2 \ar@{-}[r] & 3 \ar@{-}[r] &\cdots \ar@{-}[r] & k
\\
1 \ar@{-}[r]& 2 \ar@{-}[r] & 3 \ar@{-}[r] &\cdots \ar@{-}[r] & k
}
$$
For $k=1$ let $\theta_\gamma := \theta_1$.
Assume that $k > 1$.
If $(k-1,k) \in \Omega$, then
set $\alpha = \alpha_1 + \sum_{i=2}^{k-1} 2\alpha_i$
and $\beta = \alpha_k$, and let
$\theta_\gamma := 1/2[[\theta_\alpha,\theta_\beta],\theta_\beta]$.
It is straightforward to check that $M(\gamma)$ does not have any
filtrations of type $(\beta,\alpha,\beta)$ or $(\beta,\beta,\alpha)$.
We get $\theta_\gamma(M(\gamma)) = 1$.
If $(k,k-1) \in \Omega$, then set
$\alpha = \alpha_k$ and  $\beta = \alpha_1 + \sum_{i=2}^{k-1} 2\alpha_i$, and let
$\theta_\gamma := 1/2[\theta_\alpha,[\theta_\alpha,\theta_\beta]]$.
It is straightforward to check that $M(\gamma)$ does not have any
filtrations of type $(\alpha,\beta,\alpha)$ or $(\beta,\alpha,\alpha)$.
We get $\theta_\gamma(M(\gamma)) = 1$.

\subsubsection{}
For $1 \le k < s  \le n$ and $\gamma = \alpha_1 + \sum_{i=2}^k 2\alpha_i +
\sum_{j=k+1}^s \alpha_j$, we have
$$
M(\gamma)\df
\xymatrix{
1 \ar[d] \ar@{-}[r]& 2 \ar@{-}[r] & 3 \ar@{-}[r] &\cdots \ar@{-}[r] & k
\\
1 \ar@{-}[r]& 2 \ar@{-}[r] & 3 \ar@{-}[r] &\cdots \ar@{-}[r] & k \ar[r]
& k+1  \ar@{-}[r] & \cdots \ar@{-}[r] & s 
}
$$
if $(k+1,k) \in \Omega$,
and
$$
M(\gamma)\df
\xymatrix{
1 \ar[d] \ar@{-}[r]& 2 \ar@{-}[r] & 3 \ar@{-}[r] &\cdots \ar@{-}[r] & k
& k+1 \ar[l]\ar@{-}[r] & \cdots \ar@{-}[r] & s
\\
1 \ar@{-}[r]& 2 \ar@{-}[r] & 3 \ar@{-}[r] &\cdots \ar@{-}[r] & k
}
$$
if $(k,k+1) \in \Omega$.
For $(k+1,k) \in \Omega$,
let $\alpha = \sum_{i=k+1}^s \alpha_i$ and
$\beta = \gamma-\alpha = \alpha_1+\sum_{i=2}^k2\alpha_i$, and for
$(k,k+1) \in \Omega$, set 
$\alpha = \alpha_1 + \sum_{i=2}^k 2\alpha_i$ and 
$\beta = \gamma-\alpha = \sum_{j=k+1}^s \alpha_j$.
In both cases it is easy to check that $(\alpha,\beta,\gamma)$ is
an $H$-admissible triple.
Let
$\theta_\gamma := [\theta_\alpha,\theta_\beta]$.

\subsection{Type $F_4$}\label{primitive-F_4}
Let 
$$
C = \left(
\begin{matrix}
2&-1&&\\
-1&2&-1&\\
&-2&2&-1\\
&&-1&2
\end{matrix}
\right)
$$
be a Cartan matrix of type $F_4$.
The minimal symmetrizer of $C$ is $D = \diag(2,2,1,1)$.
Recall that a root $\gamma \in \Delta^+$ is \emph{sincere} provided
$\gamma$ is of the form $\gamma = \sum_{i=1}^4 m_i\alpha_i$
with $m_i \ge 1$ for all $i$.
For the non-sincere roots $\gamma \in \Delta^+$, we assume that $\theta_\gamma$ is already constructed (using the Dynkin types $B_3$ and
$C_3$).
There are 8 orientations for type $F_4$.
In the following Figures~\ref{F4C1}, \ref{F4C2}, \ref{F4C3}, \ref{F4C4} 
we display the Auslander-Reiten quivers
of the hereditary tensor algebras $T(C,D,\Omega)$ of type $F_4$ with 4 of these 8 orientations $\Omega$.
The other 4 orientations are treated dually.
The subgraphs of the form
$$
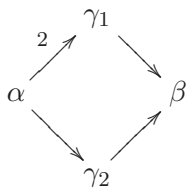
\xymatrix@-0.3cm{
& \gamma_1 \ar[dr]\\
\alpha \ar[ur]^2\ar[dr] && \beta \\
& \gamma_2 \ar[ur]
}
$$
stand for Auslander-Reiten sequences 
$$
0 \to X(\alpha) \to X(\gamma_1)^2 \oplus X(\gamma_2) \to X(\beta) \to 0.
$$
For each orientation we have 10 sincere roots $\gamma \in \Delta^+$.
Each of these will be treated separately.
So in total we need to consider 40 different cases. 
In each case, we construct for a given $\gamma$ an
$H$-admissible triple $(\alpha,\beta,\gamma)$.
For all pairs $(\alpha,\beta)$ appearing in our constructions,
one uses the techniques explained in Section~\ref{sec-preproj}
to check that $\Hom_H(M(\alpha),M(\beta)) = 0$ and
$\Hom_H(M(\beta),M(\alpha)) = 0$.
This is left to the reader.

\begin{landscape}

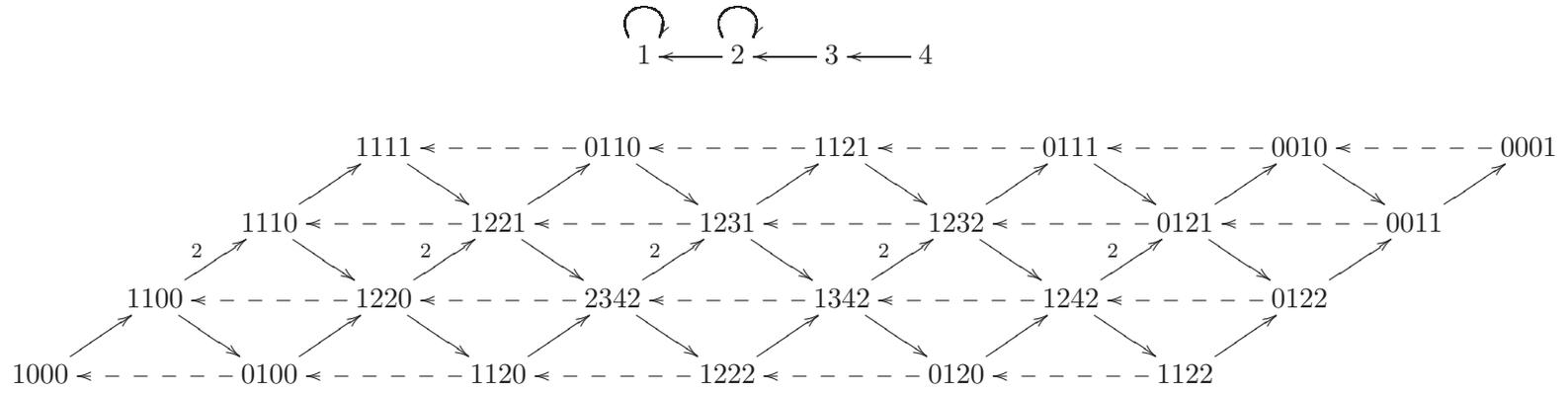
\begin{figure}
$$
\xymatrix{
1 \ar@(ul,ur)^{}& \ar[l] 2\ar@(ul,ur)^{} & \ar[l] 3 & \ar[l] 4 
}
$$

$$
\xymatrix@-0.3cm{
&&&  1111\ar[dr] && 0110\ar[dr]\ar@{-->}[ll] && 
1121\ar[dr]\ar@{-->}[ll] && 0111\ar[dr]\ar@{-->}[ll] && 
0010\ar[dr]\ar@{-->}[ll] && 0001\ar@{-->}[ll]
\\
&&  1110\ar[dr] \ar[ur] && 1221\ar[dr] \ar[ur]\ar@{-->}[ll] && 
1231\ar[dr] \ar[ur]\ar@{-->}[ll] && 1232\ar[dr] \ar[ur]\ar@{-->}[ll]
&& 
0121\ar[dr]\ar[ur]\ar@{-->}[ll] && 0011\ar[ur]\ar@{-->}[ll]
\\
& 1100\ar[dr] \ar[ur]^<<<<2 && 1220\ar[dr] \ar[ur]^<<<<2\ar@{-->}[ll] && 
2342\ar[dr] \ar[ur]^<<<<2\ar@{-->}[ll] && 1342\ar[dr] \ar[ur]^<<<<2\ar@{-->}[ll] 
&& 
1242\ar[dr]\ar[ur]^<<<<2\ar@{-->}[ll] &&  0122\ar[ur]\ar@{-->}[ll]
\\
1000 \ar[ur] && 0100 \ar[ur]\ar@{-->}[ll] && 
1120 \ar[ur]\ar@{-->}[ll] && 1222 \ar[ur]\ar@{-->}[ll] 
&& 
0120\ar[ur]\ar@{-->}[ll] && 1122\ar[ur]\ar@{-->}[ll] 
}
$$
\caption{Type $F_4$, Case 1}
\label{F4C1}
\end{figure}

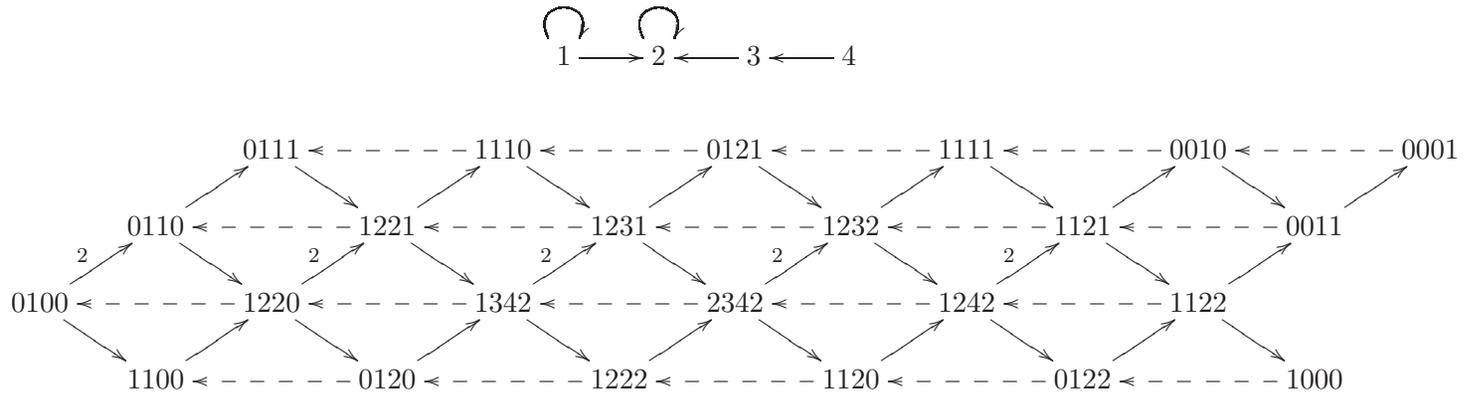
\begin{figure}
$$
\xymatrix{
1 \ar@(ul,ur)^{}\ar[r]& 2\ar@(ul,ur)^{} & \ar[l] 3 & \ar[l] 4 
}
$$

$$
\xymatrix@-0.3cm{
&&&  0111 \ar[dr] && 1110 \ar[dr]\ar@{-->}[ll] && 
0121 \ar[dr]\ar@{-->}[ll] && 1111 \ar[dr]\ar@{-->}[ll] && 
0010 \ar[dr]\ar@{-->}[ll] && 0001 \ar@{-->}[ll]
\\
&&  0110 \ar[dr] \ar[ur] && 1221 \ar[dr]\ar@{-->}[ll] \ar[ur] && 
1231\ar[dr] \ar[ur]\ar@{-->}[ll] && 1232\ar[dr] \ar[ur]\ar@{-->}[ll]
&& 
1121\ar[dr]\ar[ur]\ar@{-->}[ll] && 0011 \ar[ur]\ar@{-->}[ll]
\\
& 0100 \ar[dr] \ar[ur]^<<<<2  && 1220\ar[dr] \ar[ur]^<<<<2\ar@{-->}[ll] && 
1342\ar[dr] \ar[ur]^<<<<2\ar@{-->}[ll] && 2342\ar[dr] \ar[ur]^<<<<2\ar@{-->}[ll] 
&& 
1242\ar[dr]\ar[ur]^<<<<2\ar@{-->}[ll] && 1122 \ar[ur]\ar[dr]\ar@{-->}[ll]
\\
&& 1100 \ar[ur] && 0120 \ar[ur]\ar@{-->}[ll] && 
1222 \ar[ur]\ar@{-->}[ll] && 1120 \ar[ur]\ar@{-->}[ll] 
&& 
0122\ar[ur]\ar@{-->}[ll] &&1000\ar@{-->}[ll]
}
$$
\caption{Type $F_4$, Case 2}
\label{F4C2}
\end{figure}

\end{landscape}

\begin{landscape}

\begin{figure}
$$
\xymatrix{
1\ar@(ul,ur)^{} & \ar[l] 2\ar@(ul,ur)^{} \ar[r] & 3 & 4 \ar[l]
}
$$

$$
\xymatrix@-0.3cm{
& 0011\ar[dr] && 1110\ar[dr]\ar@{-->}[ll] && 
0121\ar[dr]\ar@{-->}[ll] && 1111\ar[dr]\ar@{-->}[ll] && 
0110\ar[dr]\ar@{-->}[ll] && 0001\ar@{-->}[ll]
\\
0010\ar[dr] \ar[ur] && 1121\ar[dr] \ar[ur]\ar@{-->}[ll] && 
1231\ar[dr] \ar[ur]\ar@{-->}[ll] && 1232\ar[dr] \ar[ur]\ar@{-->}[ll]
&& 
1221\ar[dr]\ar[ur]\ar@{-->}[ll] && 0111\ar[ur]\ar[dr]\ar@{-->}[ll]
\\
& 1120\ar[dr] \ar[ur]^<<<<2  && 1242\ar[dr] \ar[ur]^<<<<2\ar@{-->}[ll] && 
2342\ar[dr] \ar[ur]^<<<<2\ar@{-->}[ll] && 1342\ar[dr] \ar[ur]^<<<<2\ar@{-->}[ll] 
&& 
1222\ar[dr]\ar[ur]^<<<<2\ar@{-->}[ll] && 0100\ar@{-->}[ll]
\\
1000 \ar[ur] && 0120 \ar[ur]\ar@{-->}[ll] && 
1122 \ar[ur]\ar@{-->}[ll] && 1220 \ar[ur]\ar@{-->}[ll] 
&& 
0122\ar[ur]\ar@{-->}[ll] && 1100\ar[ur]\ar@{-->}[ll] 
}
$$
\caption{Type $F_4$, Case 3}
\label{F4C3}
\end{figure}
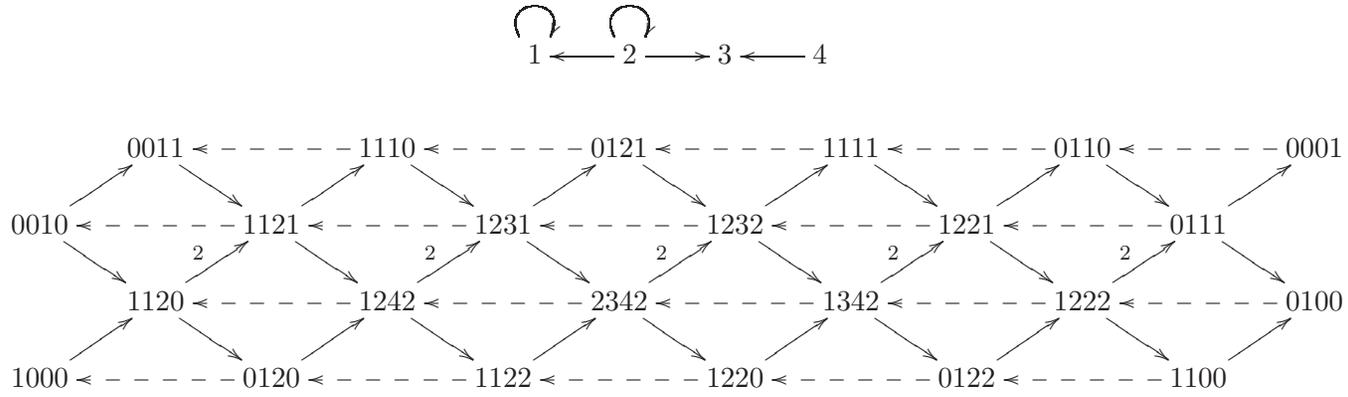

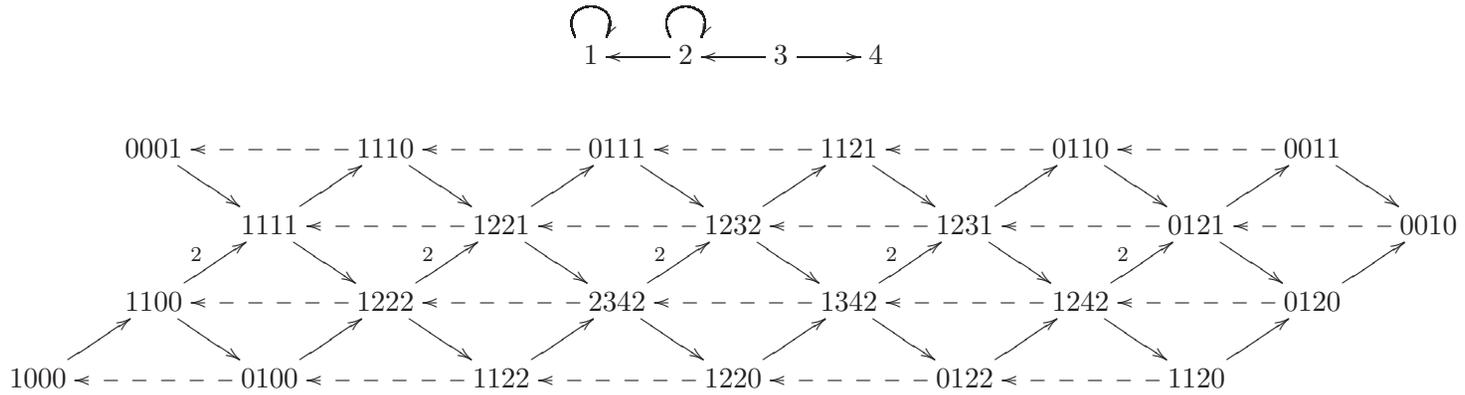
\begin{figure}
$$
\xymatrix{
1\ar@(ul,ur)^{} & \ar[l] 2\ar@(ul,ur)^{} & \ar[l] 3 \ar[r]&  4 
}
$$

$$
\xymatrix@-0.3cm{
& 0001\ar[dr] &&  1110\ar[dr]\ar@{-->}[ll] && 0111\ar[dr]\ar@{-->}[ll] && 
1121\ar[dr]\ar@{-->}[ll] && 0110\ar[dr]\ar@{-->}[ll] && 
0011\ar[dr]\ar@{-->}[ll]
\\
&&  1111\ar[dr] \ar[ur] && 1221\ar[dr] \ar[ur]\ar@{-->}[ll] && 
1232\ar[dr] \ar[ur]\ar@{-->}[ll] && 1231\ar[dr] \ar[ur]\ar@{-->}[ll]
&& 
0121\ar[dr]\ar[ur]\ar@{-->}[ll] && 0010\ar@{-->}[ll]
\\
& 1100\ar[dr] \ar[ur]^<<<<2 && 1222\ar[dr] \ar[ur]^<<<<2\ar@{-->}[ll] && 
2342\ar[dr] \ar[ur]^<<<<2\ar@{-->}[ll] && 1342\ar[dr] \ar[ur]^<<<<2\ar@{-->}[ll] 
&& 
1242\ar[dr]\ar[ur]^<<<<2\ar@{-->}[ll] && 0120\ar[ur]\ar@{-->}[ll]
\\
1000 \ar[ur] && 0100\ar[ur]\ar@{-->}[ll] && 
1122 \ar[ur]\ar@{-->}[ll] && 1220 \ar[ur]\ar@{-->}[ll] 
&& 
0122\ar[ur]\ar@{-->}[ll] && 1120\ar[ur]\ar@{-->}[ll] 
}
$$
\caption{Type $F_4$, Case 4}
\label{F4C4}
\end{figure}

\end{landscape}

If not mentioned otherwise,
in all of the following cases we define $\theta_\gamma := [\theta_\alpha,\theta_\beta]$.

In the following sections,
the 40 cases considered are labeled by $x.1,\ldots,x.10$ with $x = 1,2,3,4$.
The $x$ indicates which of the four orientations $\Omega$ displayed in
Figures~\ref{F4C1},\ref{F4C2},\ref{F4C3},\ref{F4C4} we are working with.
With respect to the diagrams displayed in 
Figures~\ref{F4C1},\ref{F4C2},\ref{F4C3},\ref{F4C4}, 
we go through the cases row-wise. 
Throughout, the logic and results of Section~\ref{adm-triples} will be used freely without reference.


\subsubsection{Case 1.1}
Let $\gamma = (1,1,1,1)$ and
$(\alpha,\beta) = ((1,1,1,0),(0,0,0,1))$.
We have
$$
M(\alpha)\df
\xymatrix@-1.8ex{
1 \ar[d] & 2 \ar[d]\ar[l] & 3 \ar[l] 
\\
1 & 2 \ar[l]
}
\hspace{1cm}
M(\gamma)\df
\xymatrix@-1.8ex{
1 \ar[d] & 2 \ar[d]\ar[l] & 3 \ar[l] & 4 \ar[l]
\\
1 & 2 \ar[l]
}
$$
There is an obvious short exact sequence
$0 \to M(\alpha) \to M(\gamma) \to M(\beta) \to 0$.
(This is not a pseudo AR-sequence.) 
It
is easily checked that
$(\alpha,\beta,\gamma)$ is an $H$-admissible triple.

\subsubsection{Case 1.2}
Let
$\gamma = (1,1,2,1)$ and
$(\alpha,\beta) = ((1,1,2,0),(0,0,0,1))$.
We have
$$
M(\alpha)\df
\xymatrix@-1.8ex{
1 \ar[d] & 2 \ar[d]\ar[l] & 3 \ar[l] 
\\
1 & 2 \ar[l] & 3 \ar[l]
}
\hspace{1cm}
M(\gamma)\df
\xymatrix@-1.8ex{
1 \ar[d] & 2 \ar[d]\ar[l] & 3 \ar[l] & 4 \ar[l]
\\
1 & 2 \ar[l] & 3 \ar[l]
}
$$
There is an obvious short exact sequence
$0 \to M(\alpha) \to M(\gamma) \to M(\beta) \to 0$.
(This is not a pseudo AR-sequence.) 
Obviously, $(\alpha,\beta,\gamma)$ is an $H$-admissible triple.

\subsubsection{Case 1.3}
Let
$\gamma = (1,2,2,1)$ and
$(\alpha,\beta) = ((1,1,1,1),(0,1,1,0))$.
There is a pseudo AR-sequence 
$0 \to M(\alpha) \to M(\gamma) \to M(\beta) \to 0$.
We have
$$
M(\beta)\df 
\xymatrix@-1.8ex{
2 \ar[d] & 3 \ar[l]\\
2 
}
$$
The only other locally free indecomposable $H$-module with rank vector
$(0,1,1,0)$ is
$$
V(\beta)\df 
\xymatrix@-1.8ex{
2 \ar[d] \\
2 & 3 \ar[l]
}
$$
We have $\Hom_H(V(\beta),E_2) \not= 0$.
Using Figure~\ref{F4C1} we get $\Hom_H(M(\gamma),E_2) = 0$.
Thus $V(\beta)$ cannot be a factor module of $M(\gamma)$.
It follows that 
$(\alpha,\beta,\gamma)$ is an $H$-admissible triple.

\subsubsection{Case 1.4}
Let $\gamma = (1,2,3,1)$ and
$(\alpha,\beta) = ((0,1,1,0),(1,1,2,1))$.
There is a pseudo AR-sequence 
$0 \to M(\alpha) \to M(\gamma) \to M(\beta) \to 0$.
We have
$$
M(\alpha)\df
\xymatrix@-1.8ex{
2 \ar[d] & 3 \ar[l] 
\\
2  
}
\hspace{1cm}
M(\beta)\df
\xymatrix@-1.8ex{
1 \ar[d] & 2 \ar[d]\ar[l] & 3 \ar[l] & 4 \ar[l]
\\
1 & 2 \ar[l] & 3 \ar[l] 
}
$$
The only other locally free indecomposable $H$-module with rank vector
$\alpha$ is 
$$
U(\alpha)\df
\xymatrix@-1.8ex{
2 \ar[d] 
\\
2  & 3 \ar[l] 
}
$$
We have $\Hom_H(U(\alpha),M(\beta)) = 0$, thus $U(\alpha)$
cannot be a submodule of $M(\gamma)$.
Thus $(\alpha,\beta,\gamma)$ is an $H$-admissible triple.

\subsubsection{Case 1.5}
Let $\gamma = (1,2,3,2)$ and
$(\alpha,\beta) = ((1,1,2,1),(0,1,1,1))$.
There is a pseudo AR-sequence 
$0 \to M(\alpha) \to M(\gamma) \to M(\beta) \to 0$.
We have
$$
M(\beta)\df
\xymatrix@-1.8ex{
2 \ar[d] & 3 \ar[l] & 4 \ar[l]
\\
2 
}
$$
The only other locally free indecomposable $H$-module with rank vector
$\beta$ is 
$$
V(\beta)\df
\xymatrix@-1.8ex{
2 \ar[d] 
\\
2  & 3 \ar[l] & 4 \ar[l]
}
$$
We have $\Hom_H(V(\beta),E_2) \not= 0$.
Using Figure~\ref{F4C1} we see that
$\Hom_H(M(\gamma),E_2) = 0$.
Thus $V(\beta)$
cannot be a factor module of $M(\gamma)$.
Thus $(\alpha,\beta,\gamma)$ is an $H$-admissible triple.

\subsubsection{Case 1.6}
Let $\gamma = (2,3,4,2)$ and $(\alpha,\beta) = ((1,1,2,0),(1,2,2,2))$.
There is a pseudo AR-sequence 
$0 \to M(\alpha) \to M(\gamma) \to M(\beta) \to 0$.
We have
$$
M(\alpha)\df
\xymatrix@-1.8ex{
1 \ar[d] & 2 \ar[d]\ar[l] & 3 \ar[l] 
\\
1 & 2 \ar[l] & 3 \ar[l] 
}
\hspace{1cm}
M(0,1,1,0)\df
\xymatrix@-1.8ex{
2 \ar[d] & 3 \ar[l] 
\\
2 
}
$$
The only other locally free indecomposable $H$-module with rank vector $\alpha$
is
$$
U(\alpha)\df
\xymatrix@-1.8ex{
1 \ar[d]
\\
1 & 2 \ar[d]\ar[l] & 3 \ar[l] 
\\
& 2  & 3 \ar[l] 
}
$$
We have $\Hom_H(M(0,1,1,0),U(\alpha)) \not= 0$.
On the other hand,
using Figure~\ref{F4C1} we get  
$\Hom_H(M(0,1,1,0),M(\gamma)) = 0$.
Thus $U(\alpha)$ cannot be a submodule of $M(\gamma)$.
Thus $(\alpha,\beta,\gamma)$ is an $H$-admissible triple.

\subsubsection{Case 1.7}
Let
$\gamma = (1,3,4,2)$ and
$(\alpha,\beta) = ((1,2,2,2),(0,1,2,0))$.
There is a pseudo AR-sequence 
$0 \to M(\alpha) \to M(\gamma) \to M(\beta) \to 0$.
The module $M(\beta)$
is the only locally free indecomposable $H$-module with rank vector $\beta$.
Thus $(\alpha,\beta,\gamma)$ is an $H$-admissible triple.

\subsubsection{Case 1.8}
Let
$\gamma = (1,2,4,2)$ and
$(\alpha,\beta) = ((0,1,2,0),(1,1,2,2))$.
There is a pseudo AR-sequence 
$0 \to M(\alpha) \to M(\gamma) \to M(\beta) \to 0$.
The module $M(\alpha)$
is the only locally free indecomposable $H$-module with rank vector $\alpha$.
Thus $(\alpha,\beta,\gamma)$ is an $H$-admissible triple.

\subsubsection{Case 1.9}
Let
$\gamma = (1,2,2,2)$ and
$(\alpha,\beta) = ((0,1,0,0),(1,1,2,2))$.
We have
$$
M(\gamma)\df
\xymatrix@-1.8ex{
1 \ar[d] & 2 \ar[d]\ar[l] & 3 \ar[l] & 4 \ar[l]
\\
1 & 2 \ar[l] & 3 \ar[l]\ar[dl] & 4 \ar[l]
\\
& 2 \ar[d]
\\
&2
}
\hspace{1cm}
M(\beta)\df
\xymatrix@-1.8ex{
1 \ar[d] & 2 \ar[d]\ar[l] & 3 \ar[l] & 4 \ar[l]
\\
1 & 2 \ar[l] & 3 \ar[l] & 4 \ar[l]
}
$$
There obviously exists a short exact sequence
$0 \to M(\alpha) \to M(\gamma) \to M(\beta) \to 0$.
(This is not a pseudo AR-sequence.)
We also see that 
there is only one locally free indecomposable $H$-module with rank vector $\alpha$,
namely $M(\alpha)$.
Thus $(\alpha,\beta,\gamma)$ is an $H$-admissible triple.

\subsubsection{Case 1.10}
Let
$\gamma = (1,1,2,2)$ and
$(\alpha,\beta) = ((1,0,0,0),(0,1,2,2))$.
We have
$$
M(\gamma)\df
\xymatrix@-1.8ex{
1 \ar[d] & 2 \ar[d]\ar[l] & 3 \ar[l] & 4 \ar[l]
\\
1 & 2 \ar[l] & 3 \ar[l] & 4 \ar[l]
}
\hspace{1cm}
M(\beta)\df
\xymatrix@-1.8ex{
2 \ar[d] & 3 \ar[l] & 4 \ar[l]
\\
2 & 3 \ar[l] & 4 \ar[l]
}
$$
There obviously exists a short exact sequence
$0 \to M(\alpha) \to M(\gamma) \to M(\beta) \to 0$.
(This is not a pseudo AR-sequence.)
There is only one locally free indecomposable $H$-module with rank vector $\alpha$.
Thus $(\alpha,\beta,\gamma)$ is an $H$-admissible triple.


\subsubsection{Case 2.1}
$\gamma = (1,1,1,1)$,
$(\alpha,\beta) = ((1,1,1,0),(0,0,0,1))$.
We have
$$
M(\alpha)\df
\xymatrix@-1.8ex{
1 \ar[r] \ar[d] & 2 \ar[d] & 3 \ar[l] \\
1 \ar[r] & 2 
}
\hspace{1cm}
M(\gamma)\df
\xymatrix@-1.8ex{
1 \ar[r] \ar[d] & 2 \ar[d] & 3 \ar[l] & 4 \ar[l]\\
1 \ar[r] & 2 
}
$$
There is a short exact sequence 
$0 \to M(\alpha) \to M(\gamma) \to M(\beta) \to 0$.
(This is not a pseudo AR-sequence.)
Obviously, $(\alpha,\beta,\gamma)$ is an $H$-admissible triple.

\subsubsection{Case 2.2}
Let $\gamma = (1,2,2,1)$ and
$(\alpha,\beta) = ((0,1,1,1),(1,1,1,0))$.
There is a pseudo AR-sequence 
$0 \to M(\alpha) \to M(\gamma) \to M(\beta) \to 0$.
We have
$$
M(\beta)\df
\xymatrix@-1.8ex{
1 \ar[d]\ar[r] & 2 \ar[d] & 3 \ar[l]
\\
1 \ar[r] & 2 
}
\hspace{1cm}
M(1,1,0,0)\df
\xymatrix@-1.8ex
{
1 \ar[d]\ar[r] & 2 \ar[d]
\\
1 \ar[r] & 2
}
\hspace{1cm}
M(0,1,2,0)\df
\xymatrix@-1.8ex
{
2 \ar[d] & 3 \ar[l]
\\
2 & 3 \ar[l]
}
$$
The only one other locally free indecomposable $H$-modules with rank vector $\beta$
are
$$
V_1(\beta)\df
\xymatrix@-1.8ex{
1 \ar[d]\ar[r] & 2 \ar[d] 
\\
1 \ar[r] & 2 & 3 \ar[l]
}
\hspace{1cm}
V_2(\beta)\df
\xymatrix@-1.8ex{
& 2 \ar[d] & 3 \ar[l]
\\
1 \ar[d]\ar[r]& 2 
\\
1
}
\hspace{1cm}
V_3(\beta)\df
\xymatrix@-1.8ex{
& 2 \ar[d] 
\\
1 \ar[r]\ar[d]& 2 & 3 \ar[l]
\\
1
}
$$
By Figure~\ref{F4C2} we have $\Hom_H(M(\gamma),M(0,1,2,0)) = 0$.
For $i=2,3$ we clearly have $\Hom_H(V_i(\beta),M(0,1,2,0)) \not= 0$.
Thus $V_i(\beta)$  cannot be a factor module of $M(\gamma)$
for $i = 2,3$.
Figure~\ref{F4C2} shows that $\Hom_H(M(\gamma),M(1,1,0,0)) = 0$.
On the other hand, we have $\Hom_H(V_1(\beta),M(1,1,0,0)) \not= 0$.
Thus $V_1(\beta)$ cannot be a factor module of $M(\gamma)$.
Thus $(\alpha,\beta,\gamma)$ is an $H$-admissible triple.

\subsubsection{Case 2.3}
Let
$\gamma = (1,2,3,1)$ and
$(\alpha,\beta) = ((1,1,1,0),(0,1,2,1))$.
There is a pseudo AR-sequence 
$0 \to M(\alpha) \to M(\gamma) \to M(\beta) \to 0$.
We have
$$
M(\alpha)\df
\xymatrix@-1.8ex{
1 \ar[d]\ar[r] & 2 \ar[d] & 3 \ar[l] 
\\
1 \ar[r] & 2 
}
\hspace{1cm}
M(\beta)\df
\xymatrix@-1.8ex{
2 \ar[d] & 3 \ar[l] & 4 \ar[l]
\\
2 & 3 \ar[l] 
}
$$
The only other locally free indecomposable $H$-modules with rank vector
$\alpha$ are
$$
U_1(\alpha)\df
\xymatrix@-1.8ex{
1 \ar[d]\ar[r] & 2 \ar[d] 
\\
1 \ar[r] & 2 & 3 \ar[l] 
}
\hspace{1cm}
U_2(\alpha)\df
\xymatrix@-1.8ex{
 & 2 \ar[d] & 3 \ar[l] 
\\
1 \ar[d]\ar[r] & 2 
\\
1
}
\hspace{1cm}
U_3(\alpha)\df
\xymatrix@-1.8ex{
 & 2 \ar[d] 
\\
1 \ar[d]\ar[r] & 2 & 3 \ar[l] 
\\
1
}
$$
For $i=2,3$
we have $\Hom_H(E_1,U_i(\alpha)) \not= 0$, and from
Figure~\ref{F4C2} we get $\Hom_H(E_1,M(\gamma)) = 0$.
Thus for $i=2,3$, $U_i(\alpha)$ cannot be a submodule of $M(\gamma)$.
We obviously have $\Hom_H(U_1(\alpha),M(\beta)) = 0$.
This implies that $U_1(\alpha)$ cannot be a submodule of $M(\gamma)$.
It follows that $(\alpha,\beta,\gamma)$ is an $H$-admissible triple.

\subsubsection{Case 2.4}
Let $\gamma = (1,2,3,2)$ and
$(\alpha,\beta) = ((0,1,2,1),(1,1,1,1))$.
There is a pseudo AR-sequence 
$0 \to M(\alpha) \to M(\gamma) \to M(\beta) \to 0$.
We have
$$
M(\alpha)\df
\xymatrix@-1.8ex{
2 \ar[d] & 3 \ar[l] & 4 \ar[l]
\\
2 & 3 \ar[l] 
}
\hspace{1cm}
M(\beta)\df
\xymatrix@-1.8ex{
1 \ar[d]\ar[r] & 2 \ar[d] & 3 \ar[l] & 4 \ar[l]
\\
1 \ar[r] & 2 
}
$$
The only other locally free indecomposable $H$-module with rank vector
$\alpha$ is
$$
U(\alpha)\df
\xymatrix@-1.8ex{
2 \ar[d] & 3 \ar[l]
\\
2 & 3 \ar[l]  & 4 \ar[l]
}
$$
We have $\Hom_H(U(\alpha),M(\beta)) = 0$.
Thus $U(\alpha)$ cannot be a submodule of $M(\gamma)$.
It follows that $(\alpha,\beta,\gamma)$ is an $H$-admissible triple.

\subsubsection{Case 2.5}
Let
$\gamma = (1,1,2,1)$ and
$(\alpha,\beta) = ((1,1,1,1),(0,0,1,0))$.
There is a pseudo AR-sequence
$0 \to M(\alpha) \to M(\gamma) \to M(\beta)
\to 0$.
The module $M(\beta)$ is the only locally free indecomposable $H$-module with rank vector $\beta$.
Thus $(\alpha,\beta,\gamma)$ is an $H$-admissible triple.

\subsubsection{Case 2.6}
Let
$\gamma = (1,3,4,2)$ and
$(\alpha,\beta) = ((0,1,2,0),(1,2,2,2))$.
There is a pseudo AR-sequence
$0 \to M(\alpha) \to M(\gamma) \to M(\beta)
\to 0$.
There is only one locally free indecomposable $H$-module with rank vector $\alpha$.
Thus $(\alpha,\beta,\gamma)$ is an $H$-admissible triple.

\subsubsection{Case 2.7}
Let $\gamma = (2,3,4,2)$ and
$(\alpha,\beta) = ((1,2,2,2),(1,1,2,0))$.
There is a pseudo AR-sequence
$0 \to M(\alpha) \to M(\gamma) \to M(\beta)
\to 0$.
We have
$$
M(\beta)\df
\xymatrix@-1.8ex{
1 \ar[d]\ar[r] & 2 \ar[d] & 3 \ar[l] 
\\
1 \ar[r] & 2 & 3 \ar[l] 
}
\hspace{1cm}
M(0,1,2,0)\df
\xymatrix@-1.8ex{
2 \ar[d] & 3 \ar[l] 
\\
2 & 3 \ar[l] 
}
$$
There is only one other locally free indecomposable $H$-module with rank
vector $\beta$, namely
$$
V(\beta)\df
\xymatrix@-1.8ex{
& 2 \ar[d] & 3 \ar[l] 
\\
1 \ar[d]\ar[r] & 2 & 3 \ar[l] 
\\
1
}
$$
Using Figure~\ref{F4C2} we get
$\Hom_H(M(\gamma),M(0,1,2,0)) = 0$.
Furthermore, we obviously have $\Hom_H(V(\beta),M(0,1,2,0)) \not= 0$.
Thus $V(\beta)$ cannot be a factor module of $M(\gamma)$.
It follows that $(\alpha,\beta,\gamma)$ is an $H$-admissible triple.

\subsubsection{Case 2.8}
Let
$\gamma = (1,2,4,2)$ and
$(\alpha,\beta) = ((1,1,2,0),(0,1,2,2))$.
There is a pseudo AR-sequence 
$0 \to M(\alpha) \to M(\gamma) \to M(\beta) \to 0$.
The module $M(\beta)$ is the only locally free indecomposable $H$-module with rank vector $\beta$.
Thus $(\alpha,\beta,\gamma)$ is an $H$-admissible triple.

\subsubsection{Case 2.9}
Let
$\gamma = (1,1,2,2)$ and
$(\alpha,\beta) = ((0,1,2,2),(1,0,0,0))$.
There is a pseudo AR-sequence 
$0 \to M(\alpha) \to M(\gamma) \to M(\beta) \to 0$.
The module $M(\beta)$ is the only locally free indecomposable $H$-module with rank vector $\beta$.
Thus $(\alpha,\beta,\gamma)$ is an $H$-admissible triple.

\subsubsection{Case 2.10} 
Let
$\gamma = (1,2,2,2)$ and
$(\alpha,\beta) = ((1,1,0,0),(0,1,2,2))$.
We have
$$
M(\alpha)\df
\xymatrix@-1.8ex{
1 \ar[d]\ar[r] & 2 \ar[d] 
\\
1 \ar[r]& 2 
}
\hspace{1cm}
M(\gamma)\df
\xymatrix@-1.8ex{
& 2 \ar[d] & 3 \ar[l] & 4 \ar[l]
\\
& 2 & 3 \ar[l]\ar[dl] & 4 \ar[l]
\\
1 \ar[d]\ar[r] & 2 \ar[d] 
\\
1 \ar[r]& 2 
}
\hspace{1cm}
M(\beta)\df
\xymatrix@-1.8ex{
& 2 \ar[d] & 3 \ar[l] & 4 \ar[l]
\\
& 2 & 3 \ar[l] & 4 \ar[l]
}
$$
There obviously exists a short exact sequence
$0 \to M(\alpha) \to M(\gamma) \to M(\beta) \to 0$.
(This is not a pseudo AR-sequence.)
The module $M(\beta)$ is the only locally free indecomposable $H$-module with
rank vector $\beta$.
Thus
$(\alpha,\beta,\gamma)$ is an $H$-admissible triple.


\subsubsection{Case 3.1}
$\gamma = (1,1,1,1)$,
$(\alpha,\beta) = ((1,1,1,0),(0,0,0,1))$.
We have
$$
M(\alpha)\df
\xymatrix@-1.8ex{
1 \ar[d] & 2 \ar[l]\ar[d] \\
1  & 2\ar[l]\ar[r] & 3 
}
\hspace{1cm}
M(\gamma)\df
\xymatrix@-1.8ex{
1 \ar[d] & 2 \ar[l]\ar[d] \\
1  & 2\ar[l]\ar[r] & 3 & 4 \ar[l]
}
$$
There obviously exists a short exact sequence
$0 \to M(\alpha) \to M(\gamma) \to M(\beta) \to 0$.
(This is not a pseudo AR-sequence.)
The module $M(\beta)$ is the only locally free indecomposable $H$-module with
rank vector $\beta$.
Thus $(\alpha,\beta,\gamma)$ is an $H$-admissible triple.

\subsubsection{Case 3.2}
Let
$\gamma = (1,1,2,1)$ and
$(\alpha,\beta) = ((0,0,1,1),(1,1,1,0))$.
There is a pseudo AR-sequence
$0 \to M(\alpha) \to M(\gamma) \to M(\beta)
\to 0$.
The module $M(\alpha)$ 
is the only locally free indecomposable $H$-module with rank vector $\alpha$.
Thus $(\alpha,\beta,\gamma)$ is an $H$-admissible triple.

\subsubsection{Case 3.3}
Let
$\gamma = (1,2,3,1)$ and
$(\alpha,\beta) = ((1,1,1,0),(0,1,2,1))$.
There is a pseudo AR-sequence 
$0 \to M(\alpha) \to M(\gamma) \to M(\beta) \to 0$.
We have
$$
M(\alpha)\df
\xymatrix@-1.8ex{
1 \ar[d] & 2 \ar[l]\ar[d] \\
1  & 2\ar[l]\ar[r] & 3 
}
\hspace{1cm}
M(\beta)\df
\xymatrix@-1.8ex{
2 \ar[d] \ar[r] & 3 & 4 \ar[l]\\
2 \ar[r] & 3 
}
$$
and
$$
M(1,1,0,0)\df
\xymatrix@-1.8ex{
1 \ar[d] & 2 \ar[l]\ar[d]\\
1  & 2 \ar[l]
}
\hspace{1cm}
M(0,1,1,0)\df
\xymatrix@-1.8ex{
2 \ar[d]\\
2 \ar[r] & 3
}
$$
There are only three more locally free indecomposable $H$-modules with rank
vector $\alpha$, namely
$$
U_1(\alpha)\df
\xymatrix@-1.8ex{
1 \ar[d] & 2 \ar[l]\ar[d] \ar[r] & 3 \\
1  & 2\ar[l]
}
\hspace{1cm}
U_2(\alpha)\df
\xymatrix@-1.8ex{
1 \ar[d]\\
1 & 2 \ar[l]\ar[d] \ar[r] & 3 \\
& 2
}
\hspace{1cm}
U_3(\alpha)\df
\xymatrix@-1.8ex{
1 \ar[d]\\
1  & 2 \ar[l]\ar[d] \\
& 2\ar[r] & 3 
}
$$
For $i=1,2$
we have $\Hom_H(M(1,1,0,0),U_i(\alpha)) \not= 0$, and from
Figure~\ref{F4C3} we know that $\Hom_H(M(1,1,0,0),M(\gamma)) = 0$.
Thus for $i=1,2$, $U_i(\alpha)$ cannot be a submodule of $M(\gamma)$.
One easily checks that $\Hom_H(U_3(\alpha),M(\beta)) = 0$.
Thus $U_3(\alpha)$ cannot be a submodule of $M(\gamma)$.
It follows that 
$(\alpha,\beta,\gamma)$ is an $H$-admissible triple.

\subsubsection{Case 3.4}
Let
$\gamma = (1,2,3,2)$ and
$(\alpha,\beta) = ((0,1,2,1),(1,1,1,1))$.
There is a pseudo AR-sequence 
$0 \to M(\alpha) \to M(\gamma) \to M(\beta) \to 0$.
We have
$$
M(\alpha)\df
\xymatrix@-1.8ex{
2 \ar[d]\ar[r] & 3 & 4 \ar[l]
\\
2 \ar[r] & 3
}
\hspace{1cm}
M(\beta)\df
\xymatrix@-1.8ex{
1 \ar[d]& 2 \ar[d]\ar[l]
\\
1 & 2 \ar[l]\ar[r] & 3 & 4 \ar[l]
}
$$
The only other locally free indecomposable $H$-module with rank vector $\alpha$
is
$$
U(\alpha)\df
\xymatrix@-1.8ex{
2 \ar[d]\ar[r] & 3
\\
2 \ar[r] & 3 & 4 \ar[l]
}
$$
We have $\Hom_H(U(\alpha),M(\beta)) = 0$.
Thus $U(\alpha)$ cannot be a submodule of $M(\gamma)$.
So $(\alpha,\beta,\gamma)$ is an $H$-admissible triple.

\subsubsection{Case 3.5}
Let
$\gamma = (1,2,2,1)$ and
$(\alpha,\beta) = ((1,1,1,1),(0,1,1,0))$.
There is a pseudo AR-sequence 
$0 \to M(\alpha) \to M(\gamma) \to M(\beta) \to 0$.
We have
$$
M(\alpha)\df
\xymatrix@-1.8ex{
1 \ar[d]& 2 \ar[d]\ar[l]
\\
1 & 2 \ar[l]\ar[r] & 3 & 4 \ar[l]
}
\hspace{1cm}
M(\beta)\df
\xymatrix@-1.8ex{
2 \ar[d] 
\\
2\ar[r] & 3
}
$$
There are only three more locally free indecomposable $H$-modules with rank
vector $\alpha$, namely
$$
U_1(\alpha)\df
\xymatrix@-1.8ex{
1 \ar[d]& 2 \ar[d]\ar[l]\ar[r] & 3 & 4 \ar[l]
\\
1 & 2 \ar[l]
}
\hspace{0.5cm}
U_2(\alpha)\df
\xymatrix@-1.8ex{
1 \ar[d]\\
1 & 2 \ar[d]\ar[l]\ar[r] & 3 & 4 \ar[l]
\\
& 2 
}
\hspace{0.5cm}
U_3(\alpha)\df
\xymatrix@-1.8ex{
1 \ar[d]\\
1 & 2 \ar[d]\ar[l]
\\
& 2 \ar[r] & 3 & 4 \ar[l]
}
$$
One easily sees that $\Hom_H(U_i(\alpha),M(\beta)) = 0$ for
$i=1,2,3$.
Thus none of the $U_i(\alpha)$ can be a submodule of $M(\gamma)$.
Thus
$(\alpha,\beta,\gamma)$ is an $H$-admissible triple.

\subsubsection{Case 3.6}
Let $\gamma = (1,2,4,2)$ and 
$(\alpha,\beta) = ((0,1,2,0),(1,1,2,2))$.
There exists a pseudo AR-sequence $0 \to M(\alpha) \to M(\gamma)
\to M(\beta) \to 0$.
There exists only one locally free indecomposable $H$-module with rank vector
$\alpha$.
Thus
$(\alpha,\beta,\gamma)$ is an $H$-admissible triple.

\subsubsection{Case 3.7}
Let
$\gamma = (2,3,4,2)$ and
$(\alpha,\beta) = ((1,1,2,2),(1,2,2,0))$.
There is a pseudo AR-sequence 
$0 \to M(\alpha) \to M(\gamma) \to M(\beta) \to 0$.
We have
$$
M(\alpha)\df
\xymatrix@-1.8ex{
1 \ar[d]& 2 \ar[d]\ar[l]\ar[r] & 3 & 4 \ar[l]
\\
1 & 2 \ar[l]\ar[r] & 3 & 4 \ar[l]
}
\hspace{1cm}
M(0,1,2,2)\df
\xymatrix@-1.8ex{
2 \ar[d]\ar[r] & 3 & 4 \ar[l]
\\
2 \ar[r] & 3 & 4 \ar[l]
}
$$
The only other locally free indecomposable $H$-module with rank vector $\alpha$
is
$$
U(\alpha)\df
\xymatrix@-1.8ex{
1 \ar[d]\\
1 & 2 \ar[d]\ar[l]\ar[r] & 3 & 4 \ar[l]
\\
& 2 \ar[r] & 3 & 4 \ar[l]
}
$$
Using Figure~\ref{F4C3} we get $\Hom_H(M(0,1,2,2),M(\gamma)) = 0$.
On the other hand
we have $\Hom_H(M(0,1,2,2),U(\alpha)) \not= 0$.
Thus $U(\alpha)$ cannot 
be a submodule of $M(\gamma)$.
Thus
$(\alpha,\beta,\gamma)$ is an $H$-admissible triple.

\subsubsection{Case 3.8}
Let
$\gamma = (1,3,4,2)$ and
$(\alpha,\beta) = ((1,2,2,0),(0,1,2,2))$.
There is a pseudo AR-sequence 
$0 \to M(\alpha) \to M(\gamma) \to M(\beta) \to 0$.
The module $M(\beta)$ is 
the only locally free indecomposable $H$-module with rank vector $\beta$.
Thus $(\alpha,\beta,\gamma)$ is an $H$-admissible triple.

\subsubsection{Case 3.9}
Let
$\gamma = (1,2,2,2)$ and
$(\alpha,\beta) = ((0,1,2,2),(1,1,0,0))$.
There is a pseudo AR-sequence 
$0 \to M(\alpha) \to M(\gamma) \to M(\beta) \to 0$.
The module $M(\alpha)$ is 
the only locally free indecomposable $H$-module with rank vector $\alpha$.
Thus $(\alpha,\beta,\gamma)$ is an $H$-admissible triple.

\subsubsection{Case 3.10}
$\gamma = (1,1,2,2)$,
$(\alpha,\beta) = ((1,0,0,0),(0,1,2,2))$.
We have
$$
M(\gamma)\df
\xymatrix@-1.8ex{
1 \ar[d] & 2 \ar[l]\ar[d]\ar[r] & 3 & 4 \ar[l]\\
1  & 2\ar[l]\ar[r] & 3 & 4 \ar[l]
}
\hspace{1cm}
M(\beta)\df
\xymatrix@-1.8ex{
2 \ar[d]\ar[r] & 3 & 4 \ar[l]\\
2\ar[r] & 3 & 4 \ar[l]
}
$$
There obviously exists a short exact sequence
$0 \to M(\alpha) \to M(\gamma) \to M(\beta) \to 0$.
(This is not a pseudo AR-sequence.)
The module $M(\alpha)$ is the only locally free indecomposable $H$-module with
rank vector $\alpha$.
Thus $(\alpha,\beta,\gamma)$ is an $H$-admissible triple.


\subsubsection{Case 4.1}
Let
$\gamma = (1,1,2,1)$ and
$(\alpha,\beta) = ((0,0,0,1),(1,1,2,0))$.
We have
$$ 
M(\gamma)\df
\xymatrix@-1.8ex{
1 \ar[d] & 2\ar[l] \ar[d] & 3\ar[l]
\\
1 & 2\ar[l]  & 3\ar[l]\ar[r] & 4
}
\hspace{1cm}
M(\beta)\df
\xymatrix@-1.8ex{
1 \ar[d] & 2\ar[l] \ar[d] & 3\ar[l]
\\
1 & 2\ar[l]  & 3\ar[l]
}
$$
There obviously exists a short exact sequence
$0 \to M(\alpha) \to M(\gamma) \to M(\beta) \to 0$.
(This is not a pseudo AR-sequence.)
The module $M(\alpha)$ is the only locally free indecomposable $H$-module with
rank vector $\alpha$.
Thus $(\alpha,\beta,\gamma)$ is an $H$-admissible triple.

\subsubsection{Case 4.2}
Let
$\gamma = (1,1,1,1)$ and
$(\alpha,\beta) = ((0,0,0,1),(1,1,1,0))$.
There is a pseudo AR-sequence 
$0 \to M(\alpha) \to M(\gamma) \to M(\beta) \to 0$.
The module $M(\alpha)$ is the only locally free indecomposable $H$-module with rank vector $\alpha$.
Thus $(\alpha,\beta,\gamma)$ is an $H$-admissible triple.

\subsubsection{Case 4.3}
Let $\gamma = (1,2,2,1)$ and
$(\alpha,\beta) = ((1,1,1,0),(0,1,1,1))$.
There is a pseudo AR-sequence 
$0 \to M(\alpha) \to M(\gamma) \to M(\beta) \to 0$.
We have
$$
M(\alpha)\df
\xymatrix@-1.8ex{
1 \ar[d] & 2 \ar[d]\ar[l] & 3 \ar[l]
\\
1 & 2 \ar[l]
}
\hspace{1cm}
M(\beta)\df
\xymatrix@-1.8ex{
2 \ar[d] & 3 \ar[l]\ar[r] & 4
\\
2 
}
$$
There is only one more locally free indecomposable $H$-module with dimension vector $\beta$, namely
$$
V(\beta)\df
\xymatrix@-1.8ex{
2 \ar[d] 
\\
2 & 3 \ar[l]\ar[r] & 4
}
$$
We have $\Hom_H(V(\beta),E_2) \not= 0$.
By Figure~\ref{F4C4} we have $\Hom_H(M(\gamma),E_2) = 0$.
Thus $V(\beta)$ cannot be a factor module of $M(\gamma)$.
This shows that
$(\alpha,\beta,\gamma)$ is an $H$-admissible triple.

\subsubsection{Case 4.4}
Let $\gamma = (1,2,3,2)$ and
$(\alpha,\beta) = ((0,1,1,1),(1,1,2,1))$.
There is a pseudo AR-sequence 
$0 \to M(\alpha) \to M(\gamma) \to M(\beta) \to 0$.
We have
$$
M(\alpha)\df
\xymatrix@-1.8ex{
2 \ar[d]& 3 \ar[l]\ar[r] & 4
\\
2 
}
\hspace{1cm}
M(\beta)\df
\xymatrix@-1.8ex{
1 \ar[d] & 2 \ar[d]\ar[l] & 3 \ar[l]
\\
1 & 2 \ar[l]& 3 \ar[l]\ar[r] & 4
}
$$
There is only one other locally free indecomposable $H$-module with rank vector
$\alpha$, namely
$$
U(\alpha)\df
\xymatrix@-1.8ex{
2 \ar[d] 
\\
2 & 3 \ar[l]\ar[r] & 4
}
$$
We have $\Hom_H(U(\alpha),M(\beta)) = 0$.
Thus $U(\alpha)$ cannot be a submodule of $M(\gamma)$.
Thus
$(\alpha,\beta,\gamma)$ is an $H$-admissible triple.

\subsubsection{Case 4.5}
Let $\gamma = (1,2,3,1)$ and
$(\alpha,\beta) = ((1,1,2,1),(0,1,1,0))$.
There is a pseudo AR-sequence 
$0 \to M(\alpha) \to M(\gamma) \to M(\beta) \to 0$.
We have
$$
M(\beta)\df
\xymatrix@-1.8ex{
2 \ar[d] & 3 \ar[l]
\\
2 
}
$$
There is only one other locally free indecomposable $H$-module with rank vector
$\beta$, namely
$$
V(\beta)\df
\xymatrix@-1.8ex{
2 \ar[d] 
\\
2 & 3 \ar[l]
}
$$
By Figure~\ref{F4C4} we have $\Hom_H(M(\gamma),E_2) = 0$.
We have $\Hom_H(V(\beta),E_2) \not= 0$.
Thus $V(\beta)$ cannot be a factor module of $M(\gamma)$.
So
$(\alpha,\beta,\gamma)$ is an $H$-admissible triple.

\subsubsection{Case 4.6}
Let $\gamma = (1,2,2,2)$ and
$(\alpha,\beta) = ((0,1,0,0),(1,1,2,2))$.
There is a pseudo AR-sequence 
$0 \to M(\alpha) \to M(\gamma) \to M(\beta) \to 0$.
The only locally free indecomposable $H$-module with rank vector $\alpha$
is $M(\alpha)$.
Thus $(\alpha,\beta,\gamma)$ is an $H$-admissible triple.

\subsubsection{Case 4.7}
Let $\gamma = (2,3,4,2)$ and
$(\alpha,\beta) = ((1,1,2,2),(1,2,2,0))$.
There is a pseudo AR-sequence 
$0 \to M(\alpha) \to M(\gamma) \to M(\beta) \to 0$.
We have
$$
M(\alpha)\df
\xymatrix@-1.8ex{
1 \ar[d] & 2 \ar[d]\ar[l] & 3 \ar[l]\ar[r] & 4
\\
1 & 2 \ar[l]& 3 \ar[l]\ar[r] & 4
}
\hspace{1cm}
M(0,1,1,0)\df
\xymatrix@-1.8ex{
2 \ar[d] & 3 \ar[l]
\\
2 
}
$$
There is only one other locally free indecomposable $H$-module with rank vector
$\alpha$, namely
$$
U(\gamma)\df
\xymatrix@-1.8ex{
1 \ar[d]
\\
1 & 2 \ar[d]\ar[l] & 3 \ar[l]\ar[r] & 4
\\
& 2 & 3 \ar[l]\ar[r] & 4
}
$$
We see from Figure~\ref{F4C4} that $\Hom_H(M(0,1,1,0),M(\gamma)) = 0$.
On the other hand
we have $\Hom_H(M(0,1,1,0),U(\alpha)) \not= 0$.
Thus $U(\alpha)$ cannot be a submodule of $M(\gamma)$.
So
$(\alpha,\beta,\gamma)$ is an $H$-admissible triple.

\subsubsection{Case 4.8}
Let
$\gamma = (1,3,4,2)$ and
$(\alpha,\beta) = ((1,2,2,0),(0,1,2,2))$.
There is a pseudo AR-sequence 
$0 \to M(\alpha) \to M(\gamma) \to M(\beta) \to 0$.
The module $M(\beta)$ is the only locally free indecomposable $H$-module with rank vector $\beta$.
Thus
$(\alpha,\beta,\gamma)$ is an $H$-admissible triple.

\subsubsection{Case 4.9}
Let $\gamma = (1,2,4,2)$ and
$(\alpha,\beta) = ((0,1,2,2),(1,1,2,0))$.
There is a pseudo AR-sequence 
$0 \to M(\alpha) \to M(\gamma) \to M(\beta) \to 0$.
The module $M(\alpha)$ is
the only locally free indecomposable $H$-module with rank vector $\alpha$.
Thus
$(\alpha,\beta,\gamma)$ is an $H$-admissible triple.

\subsubsection{Case 4.10}
Let $\gamma = (1,1,2,2)$ and
$(\alpha,\beta) = ((1,0,0,0),(0,1,2,2))$.
We have
$$ 
M(\gamma)\df
\xymatrix@-1.8ex{
1 \ar[d] & 2\ar[l] \ar[d] & 3\ar[l]\ar[r] & 4
\\
1 & 2\ar[l]  & 3\ar[l]\ar[r] & 4
}
\hspace{1cm}
M(\beta)\df
\xymatrix@-1.8ex{
2 \ar[d] & 3\ar[l]\ar[r] & 4
\\
2  & 3\ar[l]\ar[r] & 4
}
$$
There obviously exits a short exact sequence
$0 \to M(\alpha) \to M(\gamma) \to M(\beta) \to 0$.
(This is not a pseudo AR-sequence.)
The module $M(\alpha)$ is the only locally free indecomposable $H$-module with
rank vector $\alpha$.
Thus $(\alpha,\beta,\gamma)$ is an $H$-admissible triple.

\subsection{Type $G_2$}\label{primitive-G_2}
Let 
$$
C = \left(\begin{matrix} 2&-1\\-3&2\end{matrix}\right)
$$ 
be a Cartan matrix of type $G_2$.
Then $D = \diag(3,1)$ is the minimal symmetrizer of $C$.
Let $\Omega = \{ (1,2) \}$.
(The case $\Omega = \{ (2,1) \}$ is done dually.)
The algebra $H = H(C,D,\Omega)$ is defined by the quiver
$$
\xymatrix{
1\ar@(ul,ur)^{\vep_1} & \ar[l]^{\alpha_{12}} 2 
}
$$
with relation $\vep_1^3 = 0$.
The Auslander-Reiten quiver of the corresponding modulated graph of
type $G_2$ is shown in Figure~\ref{G2}.
The two subgraphs of the form
$$
\xymatrix@-0.3cm{
& \gamma \ar[dr]\\
\alpha \ar[ur]^3 && \beta
}
$$
stand for Auslander-Reiten sequences 
$$
0 \to X(\alpha) \to X(\gamma)^3 \to X(\beta) \to 0.
$$
\begin{figure}
$$
\xymatrix@-0.3cm{
& (1,1) \ar[dr] && (1,2) \ar[dr]\ar@{-->}[ll] && (0,1)\ar@{-->}[ll]
\\
(1,0) \ar[ur]^<<<<3 && (2,3) \ar[ur]^<<<<3\ar@{-->}[ll] && (1,3) \ar[ur]^<<<<3\ar@{-->}[ll]
}
$$
\caption{Type $G_2$}
\label{G2}
\end{figure}
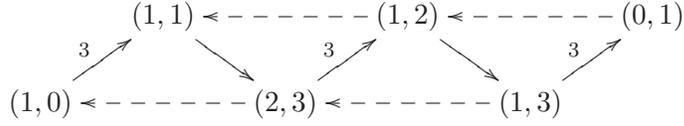

As before, for the simple roots $\alpha_i$, $i=1,2$ we define
$\theta_{\alpha_i} := \theta_i$.
Now we treat the 
remaining four positive roots
$(1,1)$, $(1,2)$, $(1,3)$, $(2,3)$.

\subsubsection{Case 1.1}
Let $\gamma = (1,1)$ and $(\alpha,\beta) = ((1,0),(0,1))$.
We have
$$
M(\gamma)\df
\xymatrix@-0.3cm{
1 \ar[d] & 2 \ar[l]\\
1 \ar[d]\\
1
}
$$
There is a short exact sequence
$0 \to M(\alpha) \to M(\gamma) \to M(\beta) \to 0$.
(This is not a pseudo AR-sequence.)
Set
$\theta_\gamma := [\theta_\alpha,\theta_\beta]$.
It follows that $\theta_\gamma(M(\gamma)) = 1$.

\subsubsection{Case 1.2}
Let $\gamma = (1,2)$ and $(\alpha,\beta) = ((1,0),(0,1))$.
We have
$$
M(\gamma)\df
\xymatrix@-0.3cm{
1 \ar[d] & 2\ar[l]\\
1 \ar[d] & 2\ar[l]\\
1
}
$$
There is a short exact sequence
$0 \to M(\alpha) \to M(\gamma) \to M(\beta)^2 \to 0$.
(This is not a pseudo AR-sequence.)
It is straightforward to check that $M(\gamma)$ does not have any
filtrations of type $(\beta,\beta,\alpha)$ or $(\beta,\alpha,\beta)$.
Set
$\theta_\gamma := 1/2[[\theta_\alpha,\theta_\beta],\theta_\beta]$.
It follows that $\theta_\gamma(M(\gamma)) = 1$.

\subsubsection{Case 1.3}
Let $\gamma = (1,3)$ and $(\alpha,\beta) = ((1,0),(0,1))$.
We have
$$
M(\gamma)\df
\xymatrix@-0.3cm{
1 \ar[d] & 2\ar[l]\\
1 \ar[d] & 2\ar[l]\\
1 & 2\ar[l]
}
$$
There is a short exact sequence
$0 \to M(\alpha) \to M(\gamma) \to M(\beta)^3 \to 0$.
(This is not a pseudo AR-sequence.)
Obviously, $M(\gamma)$ does not have any
filtrations of type $(\beta,\beta,\beta,\alpha)$, 
$(\beta,\beta,\alpha,\beta)$ or $(\beta,\alpha,\beta,\beta)$.
Set
$\theta_\gamma := 1/6[[[\theta_\alpha,\theta_\beta],\theta_\beta],\theta_\beta]$.
It follows that $\theta_\gamma(M(\gamma)) = 1$.

\subsubsection{Case 1.4}
Let $\gamma = (2,3)$ and $(\alpha,\beta) = ((1,1),(1,2))$.
We have
$$
M(\alpha)\df
\xymatrix@-0.3cm{
1\ar[d] & 2 \ar[l]
\\
1 \ar[d] 
\\
1 
}
\hspace{1cm}
M(\gamma)\df
\xymatrix@-0.3cm{
1\ar[d] & & 2 \ar@/_3ex/[ll]
\\
1 \ar[d] &1 \ar[d]& 2 \ar@/_3ex/[ll]\ar[l]
\\
1 & 1 \ar[d] & 2 \ar[l]
\\
& 1
}
\hspace{1cm}
M(\beta)\df
\xymatrix@-0.3cm{
1 \ar[d]& 2 \ar[l]
\\
1 \ar[d] & 2 \ar[l]
\\
1
}
$$
There is a pseudo AR-sequence 
$0 \to M(\alpha) \to M(\gamma) \to M(\beta) \to 0$.
Using Figure~\ref{G2} 
we get that $\Ext_H^1(M(\beta),M(\alpha))$ is $1$-dimensional.
Thus up to equivalence there exists only one non-split extension of the above form.
In contrast to all other cases studied before,
we have $\Hom_H(M(\alpha),M(\beta)) \not= 0$ and there is
more than one submodule of $M(\gamma)$ which is isomorphic to
$M(\alpha)$.
So we need to use a different strategy for this exceptional case.
We have $\theta_\alpha = [\theta_1,\theta_2]$ and
$\theta_\beta = 1/2[[\theta_1,\theta_2],\theta_2]$.
Now define $\theta_\gamma := 1/2[\theta_\alpha,\theta_\beta]$.
We get
$$
\theta_\gamma = 1/4([\theta_1,\theta_2] * [[\theta_1,\theta_2],\theta_2]
- [[\theta_1,\theta_2],\theta_2] * [\theta_1,\theta_2]).
$$
Expanding this, we get
\begin{align*}
\theta_\gamma &= 1/4(\theta_1*\theta_2*\theta_1*\theta_2*\theta_2 
- 3(\theta_1*\theta_2*\theta_2*\theta_1*\theta_2)\\ 
&\;\;\;\; + 2(\theta_1*\theta_2*\theta_2*\theta_2*\theta_1) 
+ \text{(a sum of monomials starting with $\theta_2$)}).
\end{align*}
Let $\Phi = (\beta_1,\ldots,\beta_5)$ with
$\beta_i \in \{ \alpha_1,\alpha_2\}$.
Then one easily checks that $M(\gamma)$ has a filtration of type $\Phi$ only if $\Phi = (\alpha_1,\alpha_2,\alpha_1,\alpha_2,\alpha_2)$ or
$\Phi = (\alpha_1,\alpha_1,\alpha_2,\alpha_2,\alpha_2)$.
Thus we have 
$$
\theta_\gamma(M(\gamma)) = 1/4(\theta_1*\theta_2*\theta_1*\theta_2*\theta_2)(M(\gamma)) .
$$
Now one calculates that
$\theta_\gamma(M(\gamma)) = 1$.


\section{Proof of the main result}
\label{sec-proofmain}


\subsection{The main result}
The following theorem is our main result. 
It is an analogue of \cite[Theorem 4.7]{S} for our algebra 
$H(C,D,\Omega)$, where $C$ is of Dynkin type.

\begin{Thm}\label{thm-main}
Let $H = H(C,D,\Omega)$, $\cM = \cM(H)$ and $\n = \n(C)$.
Assume that $C$ is of Dynkin type.
Then the following hold:
\begin{itemize}

\item[(i)]
The Lie algebra 
$\PP(\cM)$ is isomorphic to $\n$.

\item[(ii)]
The homomorphism $\eta_H\df U(\n) \to \M$ is an isomorphism of Hopf algebras.
 
\end{itemize}
\end{Thm}

\begin{Conj}\label{conj-main}
Theorem~\ref{thm-main} remains true
for all symmetrizable generalized Cartan matrices $C$ and
all symmetrizers $D$ of $C$.
\end{Conj}

We tried to prove Conjecture~\ref{conj-main} by
adapting Schofield's \cite{S} approach to our setup.
However we couldn't find an analogue for some delicate steps 
in his proof.

\subsection{Proof of Theorem~\ref{thm-main}}
Let $C$ be a symmetrizable generalized Cartan matrix, and let
$D$ be a symmetrizer of $C$.
Furthermore, let $\Omega$ be an orientation of $C$.
For $C$ symmetric and $D$ the identity matrix we know already from
\cite{S} that $\eta_H\df U(\n) \to \cM$ is an isomorphism.
This covers the Dynkin types $A_n$, $D_n$, $E_6$, $E_7$ and $E_8$.

Recall from Corollary~\ref{cor-phi} that there is a surjective Hopf algebra
homomorphism $\eta_H\df U(\n) \to \cM$ which maps $e_i$ to 
$\theta_i$ for all $i$.
The homomorphism $\eta$ restricts to a surjective Lie algebra homomorphism
$\n \to \cP(\cM)$ and to a surjective linear map
$\eta_{H,\gamma}\df \n_\gamma \to \cP(\cM)_\gamma$ of weight spaces.
(The positive roots are the weights of $\n$.)
If $C$ is of Dynkin type, then all weight spaces $\n_\gamma$ are
$1$-dimensional. 
Now for $C$ Dynkin and $D$ minimal, Theorem~\ref{thm-primitive}
implies that $\eta_{H,\gamma}$ is an isomorphism for all $\gamma$.
Thus $\eta_H$ restricts to a Lie algebra isomorphism $\n \to \cP(\cM)$.
But this implies that $\eta_H$ itself has to be an algebra isomorphism.
This proves Theorem~\ref{thm-main} for $D$ minimal symmetrizer.

Let now $D$ be minimal and fix some $k \ge 2$.
Without loss of generality we can assume that $C$ is connected.
(This is equivalent to $Q(C,\Omega)$ being connected.)
For $l\ge 2$, set $H = H(C,D,\Omega)$ and $H(l) := H(C,lD,\Omega)$.
Recall that, writing $Z(l):=K[\epsilon]/(\epsilon^l)$, we have that
$H(l)$ is a $Z(l)$-algebra which is free as a $Z(l)$-module.
In \cite[Section~2]{GLS2} we construct a functor 
\[
R_\vp(l)\df \repvp(H(l)) \to \repvp(H(l-1)),\quad M \mapsto M/\epsilon^{l-1}M, \qquad (l\ge 2).
\]
By \cite[Lemma~2.1 and Proposition~2.2]{GLS2} the functor $R_\vp(l)$ preserves rank vectors and yields a
bijection between the isomorphism classes of rigid locally free $H(l)$-modules and the
isomorphism classes of rigid locally free $H(l-1)$-modules.
Suppose $M(k)$ is a rigid locally free $H(k)$-module, and
let 
$$
M := (R_\vp(2) \circ \cdots \circ R_\vp(k))(M(k))
$$ 
be the corresponding rigid locally free $H$-module.
Then we have $\chi(\cFl_{M(k),\bi}) = \chi(\cFl_{M,\bi})$
for all sequences $\bi$, see \cite[Corollary~1.3]{GLS2}.

We proved already that
we have a Hopf algebra isomorphism $\eta_H\df U(\n) \to \cM(H)$.
We also have a surjective Hopf algebra homomorphism 
$\eta_{H(k)}\df U(\n) \to \cM(H(k))$ by Corollary~\ref{cor-phi}.
This yields a surjective Hopf algebra homomorphism $\psi\df \cM(H) \to \cM(H(k))$ sending $\theta_\bi := \theta_{i_1} \cdots \theta_{i_t}$
to  $\theta_\bi' := \theta_{i_1}' \cdots \theta_{i_t}'$ with $\theta_i' := \psi(\theta_i) = \eta_{H(k)}(e_i)$ for all sequences
$\bi = (i_1,\ldots,i_t)$ with $1 \le i_j \le n$ and $t \ge 1$.
$$
\xymatrix{
U(\n) \ar[r]^{\eta_H}\ar[d]_{\eta_{H(k)}} & \cM(H) \ar[dl]^{\psi}
\\
\cM(H(k))
}
$$
Let $\alpha \in \Delta^+(C)$ be a positive root, and 
let $M_H(\alpha)$ (\resp $M_{H(k)}(\alpha)$) be the corresponding indecomposable
rigid locally free $H$-module (resp. $H(k)$-module) with rank vector $\alpha$.
By Theorem~\ref{thm-primitive}
there exists an element 
$\theta_\alpha \in \cP(\cM(H))_\alpha \subset \cM(H)$ such that
$\theta_\alpha(M_H(\alpha)) = 1$.
Define $\theta_\alpha' := \psi(\theta_\alpha)$.
We know that $\theta_\alpha$ is of the form
$$
\theta_\alpha = \sum_\bi \lambda_\bi \theta_\bi
$$
for some $\lambda_\bi \in K$.
Recalling (\ref{eq.product}), we get
\begin{align*}
\theta_\alpha(M_H(\alpha)) &= \sum_\bi \lambda_\bi \theta_\bi(M_H(\alpha))
= \sum_\bi \lambda_\bi \chi(\cFl_{M_H(\alpha),\bi}) 
\\
&= \sum_\bi \lambda_\bi \chi(\cFl_{M_{H(k)}(\alpha),\bi})
= \sum_\bi \lambda_\bi \theta_\bi'
= \theta_\alpha'(M_{H(k)}(\alpha)).
\end{align*}
This implies that $\theta_\alpha'$ is a non-zero element in the root
space $\cP(\cM(H(k)))_\alpha$.
This finishes the proof.


\section{Examples}\label{sec-examples}


\subsection{A PBW-basis for Dynkin type $B_2$}\label{examples-1}
Let 
\[
C = \bbm 2&-1\\-2&2 \ebm
\]
with symmetrizer $D = \diag(2,1)$ and
$\Omega = \{ (1,2) \}$.
Thus $C$ is a Cartan matrix of Dynkin type $B_2$.
We have $f_{12} = 1$ and $f_{21} = 2$.
Then $H = H(C,D,\Omega)$ is given by the quiver
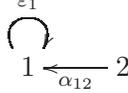
\[
\xymatrix{
1 \ar@(ul,ur)^{\vep_1} & 2 \ar[l]^{\alpha_{12}}
}
\]
with relation $\vep_1^2 = 0$. 
There are 5 isomorphism classes of indecomposable locally free $H$-modules, namely 
\[
E_1 = P_1\df
\xymatrix@-0.4cm{
1\ar[d]\\1
},
\qquad
P_2\df
\xymatrix@-0.4cm{
1\ar[d]&2\ar[l]\\1
},
\qquad
I_1\df
\xymatrix@-0.4cm{
1\ar[d]&2\ar[l]\\1&2\ar[l]
},
\qquad
E_2 = I_2\df
\xymatrix@-0.4cm{2},
\qquad
X\df
\xymatrix@-0.4cm
{1\ar[d]\\ 
1&2\ar[l]}
\]
Note that $P_2$ and $X$ have the same rank vector.
We have
\begin{align*}
\theta_{(1,0)} &= \theta_1 = \bun_{E_1}, & 
\theta_{(0,1)} &= \theta_2 = \bun_{E_2},\\
\theta_{(1,1)} &= [\theta_1,\theta_2] = \bun_{P_2} + \bun_X, &
\theta_{(1,2)} &= 1/2[[\theta_1,\theta_2],\theta_2] = \bun_{I_1}.
\end{align*}
The enveloping algebra $\M(H) \cong U(\n)$ has a Poincar\'e-Birkhoff-Witt basis given by
\[
 \theta_{(0,1)}^a * \theta_{(1,2)}^b * \theta_{(1,1)}^c * \theta_{(1,0)}^d 
 = a!b!c!d! \sum_{k=0}^c \bun_{I_2^{a}\oplus I_1^b \oplus P_2^{k}\oplus X^{c-k}\oplus P_1^d},
 \qquad (a,b,c,d\in\Z_{\ge 0}).
\]

\subsection{Pseudo AR-sequences for Dynkin type $B_2$}\label{examples-2}
Let $H = H(C,D,\Omega)$ as in Section~\ref{examples-1}.
The Auslander-Reiten quiver of 
$H$ is shown in Figure~\ref{Fig:B2}.
In the last two rows the two modules on the left have to be identified
with the corresponding two modules on the right.
The numbers stand again for basis vectors, the arrows $\alpha_{12}$ and
$\vep_1$ of $Q(C,\Omega)$ act in the obvious way, compare Section~\ref{examples-1}. 
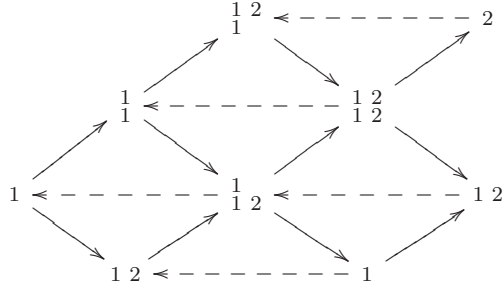
\begin{figure}[!htb]
\[
\xymatrix@R=15pt{
&&{\bsm1&2\\1\esm}\ar[dr]&&
{\bsm 2  \esm}\ar@{-->}[ll]
\\
&{\bsm 1\\1\esm}\ar[dr]\ar[ur]&&
{\bsm 1&2\\1&2\esm}\ar@{-->}[ll]\ar[dr]\ar[ur]
\\
{\bsm 1 \esm}\ar[ur]\ar[dr]&&
{\bsm 1\\1&2 \esm}\ar[dr]\ar[ur]\ar@{-->}[ll]&&
{\bsm 1&2\esm}\ar@{-->}[ll]
\\
&{\bsm 1&2 \esm}\ar[ur]&&
{\bsm 1 \esm}\ar@{-->}[ll]\ar[ur]
}
\] 
\caption{
The Auslander-Reiten quiver of $H(C,D,\Omega)$ of type $B_2$ with
$D$ minimal.}
\label{Fig:B2}
\end{figure}
There are two AR-sequences with preprojective end terms, namely
$0 \to P_2 \to I_1 \to I_2 \to 0$ and
$0 \to P_1 \to P_2 \oplus X \to I_1 \to 0$.
The first sequence is also a pseudo AR-sequences, whereas the second sequence is not.
The pseudo AR-sequence starting in $P_1$ is of the form
$0 \to P_1 \to P_2 \oplus P_2 \to I_1 \to 0$.

\bigskip
{\parindent0cm \bf Acknowledgements.}\,
We thank two anonymous referees for careful reading and detailed recommendations 
which helped to improve the exposition of this paper.
The first author acknowledges financial support from UNAM-PAPIIT grant 
IN108114 and Conacyt Grant 239255.
The third author thanks the SFB/Transregio TR 45 for financial support.



\begin{thebibliography}{999}

\bibitem[ARS]{ARS}
M. Auslander, I. Reiten, S. Smal\o,
\emph{Representation theory of Artin algebras}, 
Corrected reprint of the 1995 original. Cambridge Studies in Advanced Mathematics, 36. Cambridge University Press, Cambridge, 1997. 
xiv+425 pp. 

\bibitem[BB]{BB}
A. Bia􏲈\l{}ynicki-Birula, 
\emph{On fixed point schemes of actions of multiplicative and additive groups}, 
Topology 12 (1973), 99--102.

\bibitem[Bo]{Bo}
K. Bongartz,
\emph{A geometric version of the Morita equivalence},
J. Algebra 139 (1991), 159--171.

\bibitem[BT]{BT}
T. Bridgeland, V. Toledano Laredo,
\emph{Stability conditions and Stokes factors},
Invent. Math. 187 (2012), 61--98.

\bibitem[CBS]{CBS}
W. Crawley-Boevey, J. Schr\"oer,
\emph{Irreducible components of varieties of modules}, 
J. Reine Angew. Math. 553 (2002), 201--220.

\bibitem[DR]{DR}
V. Dlab, C.M. Ringel,
\emph{Indecomposable representations of graphs and algebras}, 
Mem. Amer. Math. Soc. 6 (1976), no. 173, v+57 pp.

\bibitem[GLS1]{GLS1}
C. Gei{\ss}, B. Leclerc, J. Schr\"oer,
\emph{Quivers with relations for symmetrizable Cartan matrices I: Foundations},
Preprint (2014), 68 pp., 
arXiv:1410.1403.

\bibitem[GLS2]{GLS2}
C. Gei{\ss}, B. Leclerc, J. Schr\"oer,
\emph{Quivers with relations for symmetrizable Cartan matrices II: 
Change of symmetrizers}, Preprint (2015), 23 pp., arXiv:1511.05898.

\bibitem[J]{J}
D. Joyce,
\emph{Configurations in abelian categories. II. Ringel-Hall algebras},
Adv. Math. 210 (2007), 635--706.

\bibitem[K1]{K_Inv}
V. Kac,
\emph{Infinite root systems, representations of graphs and invariant theory},
Invent. Math. 56 (1980), 57--92.

\bibitem[K2]{K}
V. Kac,
\emph{Infinite dimensional Lie algebras}, (3rd ed.) Cambridge 1990.

\bibitem[L1]{Lu1}
G. Lusztig,
\emph{Quivers, perverse sheaves, and quantized enveloping algebras},
J. Amer. Math. Soc. 4 (1991), 365--421.

\bibitem[L2]{Lu2}
G. Lusztig,
\emph{Semicanonical bases arising from enveloping algebras},
Adv. Math. 151 (2000), 129--139.

\bibitem[Rie]{Rie}
C. Riedtmann,
\emph{Lie algebras generated by indecomposables},
J. Algebra 170 (1994), 526--546.

\bibitem[Rin1]{R2}
C. M. Ringel,
\emph{Tame algebras and integral quadratic forms}, 
Lecture Notes in Mathematics, 1099. Springer-Verlag, Berlin, 1984. xiii+376 pp. 

\bibitem[Rin2]{R0}
C. M. Ringel,
\emph{Hall algebras and quantum groups}, 
Invent. Math. 101 (1990), 583--591.

\bibitem[Rin3]{R1}
C. M. Ringel,
\emph{Lie algebras arising in representation theory}, 
in: Representations of algebras and related topics, Kyoto, 1990.
London Math. Soc. Lecture Note Ser., 168, 284--291,
Cambridge University Press 1992.

\bibitem[S]{S}
A. Schofield,
\emph{Quivers and Kac-Moody Lie algebras},
Unpublished manuscript, 23pp.

\bibitem[Sw]{Sw}
M. E. Sweedler,
\emph{Hopf algebras},
Benjamin, New York, 1969.

\end{thebibliography}
\end{document}